\documentclass[11pt,reqno]{amsart}

\usepackage{amsthm,amssymb,amscd,graphicx,mathrsfs,pinlabel,tikz}
\usepackage[margin=1.5in]{geometry}
\usepackage{hyperref}
\usetikzlibrary{arrows}

\newtheorem{theorem}{Theorem}
\newtheorem{corollary}[theorem]{Corollary}
\newtheorem{lemma}[theorem]{Lemma}
\newtheorem{proposition}[theorem]{Proposition}

\newtheorem{innercustomthm}{Theorem}
\newenvironment{customthm}[1]
  {\renewcommand\theinnercustomthm{#1}\innercustomthm}
  {\endinnercustomthm}

\theoremstyle{definition}
\newtheorem{definition}[theorem]{Definition}
\newtheorem{definitions}[theorem]{Definitions}
\newtheorem{remark}[theorem]{Remark}

\newtheorem*{claim}{Claim} 
\newtheorem*{ack}{Acknowledgments}

\numberwithin{theorem}{section}

\newcommand{\co}{\colon \thinspace}
\newcommand{\abs}[1]{\left\lvert {#1} \right\rvert}  
\renewcommand{\leq}{\leqslant}
\renewcommand{\geq}{\geqslant}

\newcommand{\ndisks}{\chi^+}

\newcommand{\R}{{\mathbb R}}
\newcommand{\X}{{\mathbb X}} 
\newcommand{\Y}{{\mathbb Y}}
\newcommand{\Z}{{\mathbb Z}} 

\newcommand{\frakX}{\mathfrak{X}} 
\newcommand{\frakY}{\mathfrak{Y}}

\newcommand{\bu}{\mathbf{u}} 

\newcommand{\bx}{\mathbf{x}}
\newcommand{\by}{\mathbf{y}}

\newcommand{\param}{{\mathchoice{\mkern1mu\mbox{\raise2.2pt\hbox{$
          \centerdot$}}
      \mkern1mu}{\mkern1mu\mbox{\raise2.2pt\hbox{$\centerdot$}}\mkern1mu}{
      \mkern1.5mu\centerdot\mkern1.5mu}{\mkern1.5mu\centerdot\mkern1.5mu}}}

\DeclareMathOperator{\scl}{scl}
\DeclareMathOperator{\cl}{cl}
\DeclareMathOperator{\id}{id}
\DeclareMathOperator{\interior}{int}
\DeclareMathOperator{\lcm}{lcm}
\DeclareMathOperator{\PSL}{PSL}

\begin{document}

\title[Stable commutator length in Baumslag--Solitar groups]{Stable commutator length in
  Baumslag--Solitar groups and quasimorphisms for tree actions} 

\author{Matt Clay} 
\address{Department of Mathematics, University of Arkansas, Fayetteville, AR 72701, USA} 
\email{mattclay@uark.edu}

\author{Max Forester}
\address{Department of Mathematics, University of Oklahoma, Norman, OK 73019, USA}
\email{forester@math.ou.edu}

\author{Joel Louwsma}
\address{Department of Mathematics, University of Oklahoma, Norman, OK 73019, USA}
\email{jlouwsma@gmail.com}

\begin{abstract}
This paper has two parts, on Baumslag--Solitar groups and on general
$G$--trees. 

In the first part we establish bounds for stable commutator length (scl)
in Baumslag--Solitar groups. For a certain class of elements, we further show that scl is computable and takes rational
values. We also determine exactly which of these elements admit extremal
surfaces. 

In the second part we establish a universal lower bound of $1/12$ for scl
of suitable elements of any group acting on a tree. This is achieved by
constructing efficient quasimorphisms. Calculations in the
group $BS(2,3)$ show that this is the best possible universal bound, thus
answering a question of Calegari and Fujiwara. We also establish scl
bounds for acylindrical tree actions. 

Returning to Baumslag--Solitar groups, we show that their scl spectra
have a uniform gap: no element has scl in the interval $(0, 1/12)$. 
\end{abstract}

\maketitle
\thispagestyle{empty}

\section{Introduction}

Stable commutator length has been the subject of a significant amount of
recent work, especially by Danny Calegari and his collaborators.
See~\cite{Calegari:scl} for an introduction to stable commutator length
and a desciption of much of this work.  A major breakthrough in this area
was Calegari's algorithm~\cite{Calegari:free} for computing stable
commutator length in free groups.  This algorithm can also be used to
compute stable commutator length in certain classes of groups that are
built from free groups in simple ways.  However, there are few other
instances in which stable commutator length can be computed explicitly,
with the exception of certain elements and classes of groups for which it
is known to vanish.

Other work involves studying the \emph{spectrum} of values taken by
stable commutator length on a given group.  In certain cases, this
spectrum has been shown to have a \emph{gap}, i.e.\ there is a range of
values that are the stable commutator length of no element of the group. 
For example, results of this type have been shown for free
groups~\cite{DH}, for word-hyperbolic groups~\cite{CF}, and recently for
mapping class groups~\cite{BBF}. Such results
have often involved constructing quasimorphisms with certain
properties, thus relying on a dual interpretation of stable commutator
length in terms of quasimorphisms.

The primary goal of this paper is to understand stable commutator length
in Baumslag--Solitar groups. We obtain both quantitative and qualitative
results. On the way to establishing the \emph{gap theorem}
below, we digress in Section \ref{sec:qm} to construct efficient
quasimorphisms in the completely general setting of groups acting on
trees, and derive some consequences. These results may be of independent
interest to some readers. 

\subsection*{Stable commutator length in Baumslag--Solitar groups}

We use the presentation $\langle \,a, t \mid t a^m t^{-1} = a^{\ell}
\,\rangle$ for the Baumslag--Solitar group $BS(m,\ell)$, and we generally
assume that $m\not= \ell$. Then, stable commutator length is
defined exactly on the elements of $t$--exponent zero. We build on the
approach taken in \cite{BCF} and attempt to encode the computation of stable
commutator length as the output of a linear programming problem. This
approach used the notions of the \emph{turn graph} and \emph{turn circuits}
to encode the geometric data of an admissible surface. 

In the present setting, encoding this geometric data requires the use of
a \emph{weighted} turn graph instead, to account for winding numbers not 
present in the case of free groups. Even so, there is further winding
data, and the natural encoding leads to an infinite-dimensional linear
programming problem. By restricting to words of alternating $t$--shape, we
are able to reduce to a finite-dimensional problem. 

\begin{theorem}[Theorem \ref{scl-equality-alternating} and
  Corollary \ref{scl-alternating-rational}]\label{mainthm} 
Suppose $g \in BS(m,\ell)$, $m\not= \ell$, has alternating
$t$--shape. Then there is a finite-dimensional, rational linear
programming problem whose solution yields the stable commutator length of
$g$. In particular, $\scl(g)$ is computable and is a rational number. 
\end{theorem}

More generally, the linear programming problem constructed in the proof
of Theorem \ref{mainthm} is defined for any element $g$ of $t$--exponent zero,
and its solution provides a \emph{lower bound} for $\scl(g)$ (see Theorem
\ref{scl-lowerbound-lp}). What is difficult is to convert the solution
into an admissible surface to obtain a matching upper bound; the encoding
procedure from surfaces to vectors \emph{loses} information, and not
every vector can be realized by a surface. 

In some cases the solution to the linear programming problem in Theorem
\ref{mainthm} can be expressed in a closed formula. We show in
Proposition \ref{prop:length2} that if $m \nmid i$ and $\ell \nmid j$ then 
\begin{equation}\label{eq:formula}
\scl\bigl(ta^i t^{-1}a^j\bigr) \ = \ 
\frac{1}{2} \left( 1 - \frac{\gcd(i,m)}{\abs{m}} -
  \frac{\gcd(j,\ell)}{\abs{\ell}}\right).  
\end{equation}

Next we characterize the elements of alternating $t$--shape for which there
is a surface, known as an \emph{extremal surface}, that realizes the
infimum in the definition of stable commutator length. Such surfaces are
important in applications of stable commutator length to problems in
topology. It turns out that many elements have extremal surfaces, and
many do not. 

\begin{customthm}{\ref{extremal-characterization}}
  Let $g =\prod_{k=1}^r ta^{i_k} t^{-1}a^{j_k}\in BS(m,\ell)$,
  $m\neq\ell$. There is an extremal surface for $g$ if and only if
  \[
    \ell\sum_{k=1}^r i_k =-m\sum_{k=1}^r j_k.
  \] 
\end{customthm}

This allows us to find many examples of elements with rational stable
commutator length for which no extremal surface exists.  Previous
examples of this phenomenon were found in free products of abelian groups
of higher rank (see~\cite{Calegari:sails}).

Our last main result for Baumslag--Solitar groups is more qualitative in
nature and concerns the scl spectrum. 

\begin{customthm}{\ref{th:gap}}[Gap theorem]
  For every element $g \in BS(m, \ell)$, either $\scl(g) = 0$ or
  $\scl(g) \geq 1/12$. 
\end{customthm}

Thus, similar to hyperbolic groups, the spectrum has a gap above
zero. This theorem is proved in Section \ref{sec:gap}, and it depends
heavily on results in Section \ref{sec:qm} to be discussed
shortly. Nevertheless, these latter results do not apply to every element
of $BS(m,\ell)$ (namely, those that are not \emph{well-aligned}). To
study stable commutator length of these left-over elements, we take
advantage of special properties of the Bass--Serre trees for these
groups. It is interesting to note that, in contrast with Theorem
\ref{cor:acylindrical} below, it is the \emph{failure} of acylindricity
of these trees that is used in establishing the scl gap. 

\subsection*{Stable commutator length in groups acting on trees}

In order to prove the gap theorem we turn to the dual viewpoint of
quasimorphisms on groups. According to Bavard Duality
\cite{Bavard:duality}, a lower bound for $\scl(g)$ can be obtained by
finding a homogeneous quasimorphism $f$ on $G$ with $f(g) = 1$ and of
small defect. Indeed, if the defect of $f$ is $D$ then $\scl(g) \geq
1/2D$. 

Many authors have constructed quasimorphisms on groups in settings
involving negative curvature. For the most part these constructions are
variants and generalizations of the Brooks counting quasimorphisms on
free groups \cite{Rhemtulla,Brooks:qm,Grigorchuk:qm}. These settings
include hyperbolic groups \cite{EF}, groups acting on Gromov-hyperbolic
spaces \cite{Fujiwara:gromov,CF}, amalgamated free products and HNN
extensions \cite{Fujiwara:amalgam}, and mapping class groups
\cite{BF,BBF}. 

One such result is Theorem D of \cite{CF}, due to Calegari and
Fujiwara. They showed that for any amalgamated product $G = A \ast_C B$ and
any appropriately chosen hyperbolic element $g\in G$, there is a homogeneous
quasimorphism $f$ on $G$ with $f(g) = 1$ and of defect at most
$312$. This bound is of interest since it is universal, independent of
the group. 

In Theorem \ref{defect} we construct efficient quasimorphisms, of
defect at most $6$, for any group acting on a tree. These are similar to
the ``small'' counting quasimorphisms introduced by Epstein--Fujiwara
\cite{EF}, except that they are specifically tailored to the geometry of
tree actions; the counting takes place in the tree rather than
a Cayley graph. Moreover, by working directly with the homogenization of
the counting quasimorphism, we obtain a further improvement in the
defect. 

Using the calculation 
\eqref{eq:formula} in the group $BS(2,3)$ (or alternatively, a different
calculation in $\PSL(2,\Z)$) we determine that $6$ is the smallest
possible defect that can be achieved in this generality, 
thus answering Question 8.4 of \cite{CF}. Expressed in terms of
stable commutator length, the result can be stated as follows. 
\begin{customthm}{\ref{th:well-aligned}}
  Suppose $G$ acts on a simplicial tree $T$. If $g\in G$ is
  well-aligned then $\scl(g) \geq 1/12$.
\end{customthm}
The same result holds for groups acting on $\R$--trees as well (Remark
\ref{rem:rtree}).  
Again, the bound of $1/12$ is the best possible. The condition of being
\emph{well-aligned} is necessary, and agrees with the double
coset condition in \cite{CF} in the case of the Bass--Serre tree of an
amalgam. 

Not every hyperbolic element is well-aligned. Indeed, there are
examples of $3$--manifold groups that split as amalgams containing
hyperbolic elements with very small stable commutator length; see
\cite{CF}. If we consider trees that are \emph{acylindrical} (see Section
\ref{sec:qm}) then we can obtain an additional lower bound that applies to
all hyperbolic elements. This bound is almost universal, depending only
on the acylindricity constant. Alternatively, there is a genuinely
uniform bound if one considers only elements of translation length
greater than or equal to the acylindricity constant. 

\begin{customthm}{\ref{cor:acylindrical}}
  Suppose $G$ acts $K$--acylindrically on a tree $T$ and let $N$ be
  the smallest integer greater than or equal to $\frac{K}{2} + 1$. 
\begin{enumerate}
\item If $g \in G$ is hyperbolic then either $\scl(g) = 0$ or
  $\scl(g) \geq 1/12N$. 
\item If $g \in G$ is hyperbolic and $\abs{g} \geq K$ then
  either $\scl(g) = 0$ or $\scl(g) \geq 1/24$. 
\end{enumerate}
In both cases, $\scl(g) = 0$ if and only if $g$ is conjugate to
  $g^{-1}$.
\end{customthm}

\begin{ack}
  Matt Clay is partially supported by NSF grant DMS-1006898. Max
  Forester is partially supported by NSF grant DMS-1105765.
\end{ack}

\section{Preliminaries}

\subsection*{Stable commutator length}
Stable commutator length may be defined as follows, according to
Proposition 2.10 of \cite{Calegari:scl}. 

\begin{definition}\label{def1} 
Let $G = \pi_1(X)$ and suppose $\gamma \co S^1 \to X$ represents the
conjugacy class of $g\in G$. The \emph{stable commutator length} of $g$
is given by 
\begin{equation}\label{scldef} 
 \scl(g) \ = \ \inf_S \frac{-\chi(S)}{2n(S)},
\end{equation}
where $S$ ranges over all singular surfaces $S \to X$ such that 
\begin{itemize}
\item $S$ is oriented and compact with $\partial S \not= \emptyset$ 
\item $S$ has no $S^2$ or $D^2$ components 
\item the restriction $\partial S \to X$ factors through $\gamma$; that
is, there is a commutative diagram: 
\[\begin{CD}
\partial S @>>> S \\
@VVV @VVV\\
S^1 @>{\gamma}>> X
\end{CD} \]
\item the total degree, $n(S)$, of the map $\partial S \to S^1$
  (considered as a map of oriented $1$--manifolds) is non-zero. 
\end{itemize}
\end{definition}

A surface $S$ satisfying the conditions above is called an
\emph{admissible surface}. If, in addition, each component of $\partial
S$ maps to $S^1$ with \emph{positive} degree, we call $S$ a
\emph{positive admissible surface}. It is shown in Proposition 2.13 of
\cite{Calegari:scl} that the infimum in the definition of scl may be
taken over positive admissible surfaces. Such surfaces (admissible or
positive admissible) exist if and only if $g^k \in [G,G]$ for some nonzero integer $k$. If this does not occur then by convention $\scl(g) = \infty$
(the infimum of the empty set). 

A surface $S \to X$ is said to be \emph{extremal} if it realizes the
infimum in \eqref{scldef}. Notice that if this occurs, then $\scl(g)$ is a
rational number. 

\medskip

In order to bound scl from above, one needs to construct an admissible
surface realizing a given value of $\frac{-\chi(S)}{2n(S)}$. Sometimes
a procedure for building a surface cannot be completed, leaving a
surface with portions missing. The following result can be used in this
situation. 

\begin{lemma}\label{fill-with-zero}
Let $S$ be a compact oriented surface with no $S^2$ or $D^2$ components,
and whose boundary is expressed as two non-empty families of curves
$\partial_1S$ and $\partial_2 S$. Suppose $S\to X$ is a map taking the
components of $\partial_1 S$ to group elements $a_1, \ldots, a_k \in
\pi_1(X)$ and all components of $\partial_2 S$ to powers of the single
element $g \in \pi_1(X)$, with total degree $n\not= 0$. Then there is an
inequality 
\[  \scl(g) \ \leq \ \frac{-\chi(S)}{2n} + \frac{1}{n}\Bigl(\sum_i
\scl(a_i)\Bigr).\] 
\end{lemma}
More generally, if one has defined scl for chains, the sum on the right
hand side may be replaced by $\scl(\sum_i a_i)$, which may be finite even
when the original sum was not. 

\begin{proof}
  We first show how to construct a cover of $S$ that unwraps the curves
  in $\partial_1 S$ to give a collection of curves each of which is
  trivial in $H_1(X)$.  Let $b$ be the number of boundary components of
  $S$.  Let $c_i$ be the order of the conjugacy class of $a_i$ in the
  abelianization of $\pi_1(X)$.  If the conjugacy class of some $a_i$ has
  infinite order in the abelianization of $\pi_1(X)$, then
  $\scl(a_i)=\infty$ and the lemma is tautological.  Therefore we assume
  each $c_i$ is finite.  Let $M=\lcm(c_1,\dotsc,c_k)$, and consider the
  prime factorization $M=p_1^{d_1}\dotsm p_q^{d_q}$.  We construct a
  tower of covers $S_q\to S_{q-1}\to\dotsb\to S_1\to S_0=S$ as follows.
  For all $i$, the boundary $\partial S_i$ will be partitioned into two
  families of curves $\partial_1 S_i$ and $\partial_2 S_i$, where the induced map
  $S_i\to X$ takes the curves in $\partial_1 S_i$ to powers of the
  elements $a_1,\dotsc,a_k$ and the curves in $\partial_2 S_i$ to powers
  of the element $g$.  For all $i$, $\partial_1 S_i$ will consist of
  exactly $k$ curves and $\partial_2 S_i$ will consist of at least $b-k$
  curves.  

Suppose $S_{i-1}$ has been constructed.  Since $b-k\geq1$,
  there is some integer $e_i$ satisfying $k\leq e_i\leq b$ such that
  $e_i-1$ is relatively prime to $p_i$, and hence to $p_i^{q_i}$.
  Therefore Lemma 1.12 of~\cite{Calegari:scl} shows that, for any $e_i$
  boundary components of $S_{i-1}$, there is a $p_i^{d_i}$--sheeted
  covering $S_{i}\to S_{i-1}$ that unwraps these $e_i$ boundary
  components.  We choose these $e_i$ boundary components to be the $k$
  curves in $\partial_1 S_{i-1}$ and any $e_i-k$ curves in $\partial_2
  S_{i-1}$.  Then $\partial S_i$ is also partitioned into two collections
  of curves: those in the preimage of $\partial_1 S_{i-1}$ are said to be
  in $\partial_1 S_i$, and those in the preimage of $\partial_2 S_{i-1}$
  are said to be in $\partial_2 S_i$.  By construction, $\partial_1 S_i$
  consists of exactly $k$ curves and $\partial_2 S_i$ consists of at
  least $b-k$ curves. 

Iterating this procedure, we obtain a surface $S_q$
  that is a degree $M$ cover of $S$.  The induced map $S_q\to X$ takes the curves
  in $\partial_1 S_{q}$ to $a_1^M,\dotsc,a_k^M$ and the curves in
  $\partial_2 S_q$ to powers of $g$ with total degree $nM$.  Note that,
  for each $i$, $a_i^M$ is trivial in the abelianization of $\pi_1(X)$.

  Fix $\epsilon > 0$. For all $N$ relatively prime to $k-1$, we can
  construct a further cover $S_{q,N}\to S_q$ such that the curves in
  $\partial S_{q,N}$ are again partitioned into classes $\partial_1
  S_{q,N}$ and $\partial_2 S_{q,N}$, where the curves in $\partial_1
  S_{q,N}$ map to $a_1^{MN},\dotsc,a_k^{MN}$ in $X$.  Choose $N$
  sufficiently large that, for all $i$,
  the element $a_i^{MN}$ bounds an admissible surface $S'_i$ that
  approximates $\scl\bigl(a_i^{MN}\bigr)$ to within $\epsilon /k$. Since
  $\scl\bigl(a_i^{MN}\bigr)=MN\scl(a_i)$, we can also regard $S'_i$ as an
  admissible surface for $a_i$ that approximates $\scl(a_i)$ within
  $\epsilon/kMN$.  More precisely,
\[\frac{-\chi(S'_i)}{2MN} \ \leq \ \scl(a_i) + \frac{\epsilon}{kMN}\]
for each $i$. Now join the surfaces $S'_i$ along their boundaries to the
corresponding curves in $\partial_1 S_{q,N}$.  We thus obtain an
admissible surface $S''$ for $g$, with $n(S'') =nMN$. We have
\begin{align*}
  \frac{-\chi(S'')}{2n(S'')} \ &= \ \frac{-\chi(S_{q,N}) + \sum_i
    - \chi(S'_i)}{2nMN} \\
  &= \ \frac{-MN \chi(S) + \sum_i -\chi(S'_i)}{2nMN} \\
  &\leq \ \frac{-\chi(S)}{2n} + \frac{1}{n}\sum_i\Bigl(\scl(a_i) +
  \frac{\epsilon}{kNM}\Bigr) \\ 
  &= \ \frac{-\chi(S)}{2n} + \frac{1}{n}\Bigl(\sum_i\scl(a_i)\Bigr) +
  \frac{\epsilon}{nMN}.
\end{align*}
Hence $\scl(g) \ \leq \ \frac{-\chi(S)}{2n} + \frac{1}{n}\bigl(\sum_i
\scl(a_i)\bigr)$. 
\end{proof}

\subsection*{Baumslag--Solitar groups}

Before discussing Baumslag--Solitar groups per se, we make a general
observation: 

\begin{lemma}\label{scl-a}
In any group $G$, if $t$ and $a$ are elements satisfying the
Baumslag--Solitar relation $t a^m t^{-1} = a^{\ell}$ with $m \not= \ell$
then $\scl(a) = 0$. 
\end{lemma}

\begin{proof}
For any space $X$ with fundamental group $G$ there is a singular annulus
$S \to X$, whose oriented boundary components represent $a^m$ and
$a^{-\ell}$ respectively (since $a^m$ and $a^{\ell}$ are conjugate in
$G$). This surface can be made admissible with $\chi(S) = 0$ and $n(S) =
m -\ell \not= 0$, so $\scl(a) = 0$. 
\end{proof}

The \emph{Baumslag--Solitar group} $BS(m,\ell)$ is defined by the
presentation
\begin{equation}\label{bs-pres}
 \langle \, a, t \mid t a^m t^{-1} = a^{\ell} \, \rangle.
\end{equation}
The corresponding presentation $2$--complex will be denoted $X_{m,
\ell}$, or simply $X$, in this paper. One thinks of $X$ as being
constructed by attaching both ends of an annulus to a circle, by covering
maps of degrees $m$ and $\ell$ respectively; see Section \ref{sec:turn}. 

Clearly, $BS(1,1)$ is $\Z \times \Z$ and $BS(1,-1)$ is the Klein bottle
group. The cases $BS(m, \pm m)$ are also of special interest. By
constructing a suitable covering space of $X$, one finds that this group
contains a subgroup of index $2m$ isomorphic to $F_{2m-1} \times
\Z$. 
In particular, stable commutator length can be
computed in $BS(m, \pm m)$ and is always rational, using the rationality
theorem for free groups \cite{Calegari:free} and results from
\cite{Calegari:scl} (such as Proposition 2.80) on subgroups of finite
index. 

In this paper we will study stable commutator length in $BS(m, \ell)$
under the standing assumption that $m \not= \ell$. 

\begin{remark}\label{t-exp}
The abelianization of $BS(m, \ell)$ is $\Z \times \Z_{\abs{m-\ell}}$ with
generators $t$ and $a$ respectively. Since we are assuming that $m\not=
\ell$, an element of $BS(m, \ell)$ has finite order in the abelianization
if and only if it has $t$--exponent zero. Thus scl is finite on exactly
these elements. 
\end{remark}

\begin{definition}\label{t-length}
  Given a word $w$ in the letters $a^{\pm 1}$ and $t^{\pm 1}$ we
  denote by $|w|_t$ the \emph{$t$--length} of $w$.  That is, $|w|_t$
  is the number of occurrences of $t$ and $t^{-1}$ in $w$.

  Given an element $g \in BS(m,\ell)$ we denote by $|g|_t$ the
  \emph{$t$--length} of the conjugacy class of $g$.  That is, $|g|_t$
  is the minimum value of $|w|_t$ over all words $w$ that represent a
  conjugate of $g$.
\end{definition}

\begin{remark}\label{cyc-red}
  Any element $g \in BS(m,\ell)$ has a conjugate that can be expressed
  as
  \begin{equation}\label{eq:cyc-red}
    w = t^{\epsilon_1}a^{k_1}t^{\epsilon_2} \cdots t^{\epsilon_n}a^{k_n},
  \end{equation}
  where:
  \begin{itemize}
  \item $\epsilon_i \in \{1,-1\}$ for $i = 1,\ldots,n$,
  \item $m \nmid k_i$ if $\epsilon_i = 1$ and $\epsilon_{i+1} = -1$, 
  \item $\ell \nmid k_i$ if $\epsilon_i = -1$ and $\epsilon_{i+1} = 1$,
    and
  \item $|g|_t = |w|_t = n$.
  \end{itemize}
  The subscripts in the second and third bullet are read modulo $n$.
  We refer to such a representative word of the conjugacy class of $g$ as
  \emph{cyclically reduced}. 

Up to cyclic permutation, the cyclically reduced word representing a
conjugacy class is not unique. 
Two other modifications to the word \eqref{eq:cyc-red} can be made,
resulting in cyclically reduced words representing the same element:
\[ a^i t a^j  \ \leftrightarrow \  a^{i-\ell} t a^{j+m} \ \text{ and } \ a^i
t^{-1} a^j  \ \leftrightarrow \  a^{i+m} t^{-1} a^{j-\ell}. \] 
Collins' Lemma \cite{Collins,Lyndon-Schupp} characterizes precisely when
two cyclically reduced words represent the same conjugacy class. It
implies easily that modulo the two moves above and cyclic
permutation, the expression \eqref{eq:cyc-red} is unique. 
\end{remark}

\section{Surfaces in $X_{m,\ell}$}\label{sec:turn}

\subsection*{Transversality}\label{transversality} 
Transversality will be used to convert a singular admissible surface
$S \to X$ into a more combinatorial object. We will follow the approach
from \cite{BCF}, which treated the case of surfaces mapping into
graphs.

Recall that $X = X_{m, \ell}$ is the presentation $2$--complex for the
presentation \eqref{bs-pres}. We can build $X$ in the following
way. Let $A$ be the annulus $S^1 \times [-1, 1]$, and let $C$ be a space
homeomorphic to the circle. Fix orientations of $S^1$ and $C$ and attach
the boundary circles $S^1 \times \{\pm 1\}$ to $C$ via covering maps of
degrees $m$ and $\ell$ respectively, to form $X$. Note that the natural
map $\phi\co A \to X$ is surjective, and maps the interior of $A$
homeomorphically onto $X - C$. Thus we have an identification of $X-C$
with $S^1 \times (-1,1)$. 

The space $X$ is also a cell complex with $C$ as a
subcomplex. The $1$--skeleton of $X$ may be taken to be $C$ (having one
$0$--cell and one $1$--cell, labeled $a$) along with an additional
$1$--cell labeled $t$, which is a fiber in $A$ whose endpoints are
attached to the $0$--cell of $C$. 

Let $C' = S^1 \times \{0\} \subset X-C$. This is a codimension-one
submanifold. For any compact surface $S$ and continuous map $f \co S \to
X$, we may perturb $f$ by a small homotopy to make it \emph{transverse}
to $C'$. Then, $f^{-1}(C')$ is a properly embedded codimension-one
submanifold $N \subset S$. By a further homotopy, we can arrange that $N$
has an embedded $I$--bundle neighborhood $N \times [-1,1] \subset S$
(with $N = N \times \{0\}$) such that $f^{-1}(X-C) = N \times (-1, 1)$
and 
\[f\vert_{N \times (-1,1)} \co N \times (-1,1) \to S^1 \times (-1,1)\]
is a map of the form $f_0 \times \id$. 

Let $N_b \subset N$ be the union of the components that are intervals
(rather than circles). Let $S_b \subset S$ be the subset $N_b \times
[-1,1]$, each component of which is a \emph{band} $I \times [-1,1]$ with
$(I \times [-1,1]) \cap \partial S = \partial I \times [-1,1]$.  

By a further homotopy of $f$ in a neighborhood of $\partial S$, and using
transversality for the map $S - (N \times (-1,1)) \to C$, we can
arrange that in addition to the structure given so far, there is a
collar neighborhood $S_{\partial} \subset S$ on which $f$ has a simple
description. This map takes $S_{\partial}$ into the $1$--skeleton of $X$
by a retraction onto $\partial S$ followed by the restriction $\partial S
\to X$. Each annulus component of $S_{\partial}$ decomposes into squares
that retract into $\partial S$ and then map to $X$ by the characteristic
maps of $1$--cells. These squares are labeled $a$-- or $t$--squares
depending on the $1$--cell. The $t$--squares are exactly the components
of $S_{\partial} \cap S_b$.  In particular, each band ends in two
$t$-squares, representing one instance each of $t$ and $t^{-1}$ 
along the boundary. See Figure \ref{fig:surface}. 

\begin{figure}[ht]
  \centering
  \labellist
  \small\hair 2pt
  \pinlabel {$\times \, m$} [l] at 249 59
  \pinlabel {$\times \, \ell$} [r] at 343 59
  \pinlabel {$C$} [t] at 297 66
  \pinlabel {$C'$} [t] at 298 16
  \pinlabel {$t^{-1}$} [b] at 58 -5
  \pinlabel {$t$} [b] at 139 -5
  \pinlabel {$a$} [b] at 31 -5
  \pinlabel {$a^{-1}$} [b] at 87 -5
  \pinlabel {$a^{-1}$} [b] at 113 -5
  \pinlabel {$a$} [b] at 167 -5
  \endlabellist
  \includegraphics{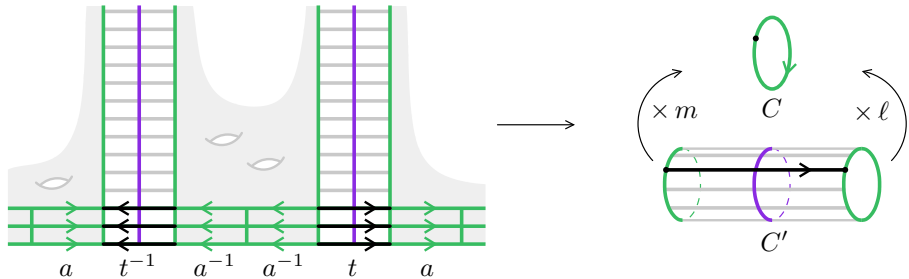}
  \caption{An admissible surface after the transversality procedure. The
    gray regions map into $C$.}\label{fig:surface} 
\end{figure}

Finally, we define $S_0 = S_{\partial} \cup S_b$ and $S_1 = S -
  \interior(S_0)$. Observe that $f$ maps $\partial S_1$ into $C$. 

The boundary $\partial S_0$ decomposes into two subsets: $\partial S$,
called the \emph{outer boundary}, and components in the interior of $S$,
called the \emph{inner boundary}, denoted $\partial^- S_0$. Note that
$\partial^- S_0 = S_0 \cap S_1 = \partial S_1$. In particular, components
of the inner boundary map by $f$ to loops in $X$ representing conjugacy
classes of powers of $a$. 

\begin{remark}\label{boundarymap}
Call a loop $f \co S^1 \to X$ \emph{regular} if $S^1$ can be decomposed into
vertices and edges such that the restriction of $f$ to each edge factors
through the characteristic map of a $1$--cell of $X$. Note that a regular
map is completely described (up to reparametrization) by a cyclic word
in the generators $a^{\pm 1}$, $t^{\pm 1}$ representing the conjugacy
class of $f$ in $\pi_1(X)$. 

If a singular surface $S \to X$ has the property that its restriction to
each boundary component is regular, then the transversality procedure
described above can be performed rel boundary, so that the cyclic
orderings of oriented $a$-- and $t$--squares in $S_{\partial}$ agree with
the cyclic boundary words one started with. 
\end{remark}

Recall that $\scl(g)$ is the infimum of $\frac{-\chi(S)}{2n(S)}$ over
all positive admissible surfaces.  We will show how to compute $\scl(g)$
using the decomposition described above.  

Choose a cyclically reduced word $w$ representing the conjugacy class of
$g$. For any positive admissible surface $S$, each boundary component
maps by a loop representing a positive power of $g$ in $\pi_1(X)$. Modify
$f$ by a homotopy to arrange that its boundary maps are regular, with
corresponding cyclic words equal to positive powers of $w$. 
Then perform the transversality procedure given above, keeping the
boundary map fixed  (cf.\ Remark \ref{boundarymap}). 
At this point, the subsurfaces $S_0$, $S_1$ are defined. Each boundary
component is labeled by a positive power of $w$ and these powers add to
$n(S)$. 


Note that $\chi(S) = \chi(S_0) + \chi(S_1)$ since $S_0$ and $S_1$ meet
along circles.  Also,
\[\chi(S_0) \ = \ \frac{-n(S) \abs{g}_t}{2},\] 
as this is exactly the number of bands in $S_b$, each band connecting
two instances of $t^{\pm 1}$ in $w^{n(S)}$ and contributing $-1$
to $\chi(S_0)$. (Note that $\chi(S_{\partial}) = 0$.) 
Let $\ndisks(S_1)$ denote
the number of disk components in $S_1$.  
We have 
\begin{equation*}
  \chi(S) \ = \ \frac{-n(S)\abs{g}_t}{2} + \chi(S_1) \ \leq \
  \frac{-n(S) \abs{g}_t}{2} + \ndisks(S_1),  
\end{equation*}
and therefore 
\begin{equation*}\label{chi-lowerbound}
  \frac{-\chi(S)}{2n(S)} \ \geq \ \frac{\abs{g}_t}{4} \ - \
  \frac{\ndisks(S_1)}{2 n(S)} .
\end{equation*}
From this, we conclude that
\begin{equation}\label{scl-lowerbound}
  \scl(g) \ \geq \ \frac{\abs{g}_t}{4} \ + \
  \inf_S \frac{-\ndisks(S_1)}{2n(S)},
\end{equation}
where the infimum is taken over all positive admissible surfaces.  
In fact, the reverse of inequality \eqref{scl-lowerbound} holds as well: 

\begin{lemma}\label{scl-g}
  There is an equality 
  \begin{equation*}\label{eq:scl-g}
    \scl(g) \ = \ \frac{\abs{g}_t}{4} \ + \
    \inf_S \frac{-\ndisks(S_1)}{2n(S)}. 
  \end{equation*}
\end{lemma}

\begin{proof}
Given an admissible surface $S \to X$ decomposed as above, let $S'$ be
the union of $S_0$ and the disk components of $S_1$. Recall that the
components of $\partial S'$ in $\partial^- S_0$ map to loops in $X$
representing conjugacy classes of powers of $a$. Thus Lemma
\ref{fill-with-zero} and Lemma \ref{scl-a} imply 
\begin{equation*}
  \scl(g) \ \leq \ \frac{-\chi(S')}{2n(S)} + \frac{1}{n(S)}\sum
  \scl(a^{p_i}) \ = \ \frac{-\chi(S')}{2n(S)}. 
\end{equation*}
Since 
\[
\frac{-\chi(S')}{2n(S)} \ = \ \frac{-\chi(S_0)}{2n(S)} -
\frac{\ndisks(S_1)}{2n(S)} 
\ = \ \frac{\abs{g}_t}{4} - \frac{\ndisks(S_1)}{2n(S)} 
\]
and $S$ was arbitrary, the reverse of inequality \eqref{scl-lowerbound}
holds, as desired. 
\end{proof}

\begin{lemma}\label{extremal-disks-annuli}
  If $S$ is an extremal surface for $g$, then $S_1$ consists only of 
  disks and annuli.
\end{lemma}
\begin{proof}
  Let $S_2$ be the union of the components of $S_1$ that have
  nonnegative Euler characteristic, and let $S_3$ be the union of the
  components of $S_1$ that have negative Euler characteristic.  Then
  $S_2$ consists only of disks and annuli and
  $\chi(S_2)=\ndisks(S_1)$.  If $S$ is extremal, we must have
  \[
  \scl(g) \ =\ \frac{-\chi(S)}{2n(S)} \ = \ 
  \frac{\abs{g}_t}{4}-\frac{\chi(S_2)}{2n(S)}-\frac{\chi(S_3)}{2n(S)} \
  = \ 
  \frac{\abs{g}_t}{4}-\frac{\ndisks(S_1)}{2n(S)}-\frac{\chi(S_3)}{2n(S)}.
  \]
  Comparing with Lemma~\ref{scl-g}, this means $\chi(S_3)\geq0$,
  meaning that $S_3$ must be empty.  Thus $S_1$ consists only of disks
  and annuli.
\end{proof}

\subsection*{The weighted turn graph}

As in \cite{BCF}, we use a graph to keep track of the combinatorics of
the inner boundary $\partial^- S_0$.  

Consider a cyclically reduced word $w$ as in \eqref{eq:cyc-red}.  A
\emph{turn} in $w$ is a subword of the form $a^k$ between two occurrences
of $t^{\pm 1}$ considered as a cyclic word.  The turns are indexed by the
numbers $i = 1,\ldots,n$; the $i^{\rm th}$ turn is labeled by the
subword $t^{\epsilon_i}a^{k_i}t^{\epsilon_{i+1}}$.  A turn labeled
$ta^kt^{-1}$ is of \emph{type m}; a turn labeled $t^{-1}a^kt$ is of
\emph{type $\ell$}; all other turns are of \emph{mixed} type.

The \emph{weighted turn graph} $\Gamma(w)$ is a directed graph with
integer weights assigned to each vertex. The vertices correspond to the
turns of $w$ and the weight associated to the $i^{\rm th}$ turn is
$k_i$. There is a directed edge from turn $i$ to turn $j$ whenever
$-\epsilon_i = \epsilon_{j+1}$. In other words, if the label of a turn begins
with $t^{\pm 1}$, then there is a directed edge from this turn to every
other turn whose label ends with $t^{\mp 1}$.  The vertices of $\Gamma(w)$
are partitioned into four subsets where the presence of a directed
edge between two vertices depends only on which subsets the
vertices lie in.  Figure \ref{fig:turngraph} shows the turn graph in
schematic form. 

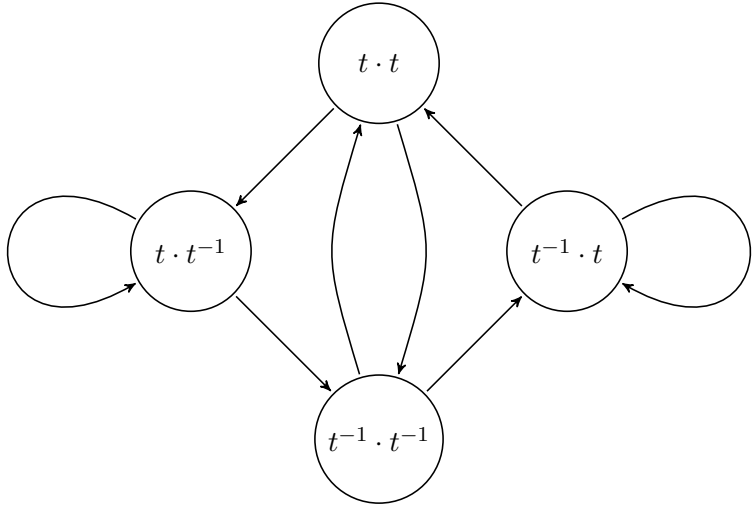
\begin{figure}[ht!]
\begin{tikzpicture}[semithick, >=stealth', shorten <=1pt, shorten >=1pt,scale=2.5]
\node [circle,draw,minimum size=1.6 cm] at (-1,0) (tT) {$t \cdot t^{-1}$};
\node [circle,draw,minimum size=1.6 cm] at (1,0) (Tt) {$t^{-1} \cdot t$};
\node [circle,draw,minimum size=1.6 cm] at (0,-1) (TT) {$t^{-1} \cdot t^{-1}$};
\node [circle,draw,minimum size=1.6 cm] at (0,1) (tt) {$t \cdot t$};
\draw[->] (tT) .. controls (-2.2,0.7) and (-2.2,-0.7) ..  (tT);
\draw[->] (Tt) .. controls (2.2,0.7) and (2.2,-0.7) .. (Tt);
\draw[->] (TT) .. controls (-0.3,0) .. (tt);
\draw[->] (tt) .. controls (0.3,0) .. (TT);
\draw[->] (tt) -- (tT);
\draw[->] (tT) -- (TT);
\draw[->] (TT) -- (Tt);
\draw[->] (Tt) -- (tt);
\end{tikzpicture}
\caption{A schematic picture of the turn graph
    $\Gamma(w)$.}\label{fig:turngraph}
\end{figure}

The edges of the turn graph come in \emph{dual pairs}: if $e \in \Gamma(w)$
is an edge from turn $i$ to turn $j$, then one verifies easily that there
is also an edge $\bar{e}$ from turn $j+1$ to turn $i-1$, and moreover
$\bar{\bar{e}} = e$.  See Figure~\ref{fig:dual-edge}.

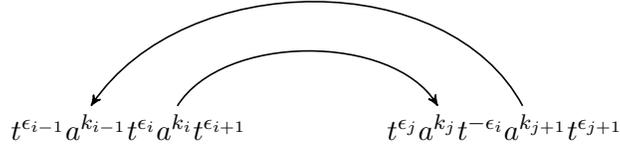
\begin{figure}[ht!]
\begin{tikzpicture}[semithick, >=stealth', shorten <=1pt, shorten >=1pt,scale=2.5]
\node at (-1,0) (i) {$t^{\epsilon_{i-1}}a^{k_{i-1}}t^{\epsilon_{i}}a^{k_{i}}t^{\epsilon_{i+1}}$};
\node at (1,0) (j) {$t^{\epsilon_{j}}a^{k_{j}}t^{-\epsilon_{i}}a^{k_{j+1}}t^{\epsilon_{j+1}}$};
\draw[->] (-0.75,0.1) .. controls (-0.5,0.5) and (0.4,0.5) .. (0.65,0.1); 
\draw[->] (1.1,0.1) .. controls (0.7,0.85) and (-0.7,0.85) .. (-1.21,0.1);
\end{tikzpicture}
\caption{Dual edge pairs in $w$. Whenever there is an edge from turn $i$
  to turn $j$, there must also be an edge from turn $j + 1$ to turn
  $i-1$. The labels of these turns share occurrences of $t^{\pm 1}$ with
  the labels of the original two turns.}\label{fig:dual-edge} 
\end{figure}

A directed circuit in $\Gamma(w)$ is of \emph{type m} or \emph{type
  $\ell$} if every vertex it visits corresponds to a turn of
type $m$ or of type $\ell$, respectively.  Otherwise, the circuit is of
\emph{mixed} type.  The \emph{weight} $\omega(\gamma)$ of a directed 
circuit $\gamma$ is the sum of the weights of the vertices it
visits (counted with multiplicity). Given a directed circuit $\gamma$,
define
\[ \mu(\gamma) \ = \ \left\{
  \begin{array}{ll}
  m & \text{ if $\gamma$ is of type $m$} \\
  \ell & \text{ if $\gamma$ is of type $\ell$} \\
  \gcd(m,\ell) & \text{ otherwise.}
\end{array}\right. \]
A directed circuit $\gamma$ is a \emph{potential disk} if $\omega(\gamma)
\equiv 0 \mod \mu(\gamma)$. 

\subsection*{Turn circuits} 
Let $S \to X$ be a positive admissible surface whose boundary map is
regular and labeled by $w^{n(S)}$. Decomposing $S$ as $S_0 \cup S_1$,
each inner boundary component of $S_0$ can be described as
follows. Traversing the curve in the positively oriented 
direction, one alternately follows the boundary arcs (or \emph{sides}) of
bands in $S_b$ and visits turns of $w$ along $S_{\partial}$; such a
visit consists in traversing the inner edges of some $a$--squares
before proceeding up along another side of a band (cf.\ Figure
\ref{fig:surface}).  If the side of the band leads from turn
$i$ to turn $j$, then $(t^{\epsilon_i})^{-1} = t^{\epsilon_j}$ and
therefore there is an edge in $\Gamma(w)$ from turn $i$ to turn $j$.
In this way, $\partial^- S_0$ gives rise to a finite collection (possibly
with repetitions) of directed circuits in $\Gamma(w)$, called the
\emph{turn circuits} for $S_0$.

Since $\partial S$ is labeled by $w^{n(S)}$, there are $n(S)$
occurrences of each turn on $\partial S$. The turn circuits do not
contain the information of which particular instances of turns are
joined bands, nor do they record how many times the band corresponding
to a given edge in the circuit wraps around the annulus $X - C$. 

\begin{remark}\label{turngraph}
  Given two cyclically reduced words $w,w'$ representing the same
  conjugacy class in $BS(m,\ell)$, there is an isomorphism $\Gamma(w)
  \to \Gamma(w')$ of the underlying directed graph structure that
  respects vertex type and edge duality but not necessarily the vertex
  weights.  However, a directed circuit is a potential disk with one
  sets of weights if and only if it is a potential disk with the other
  set.  The difference in weights of a type $m$ vertex is a multiple
  of $m$, the difference in weights of a type $\ell$ vertex is a
  multiple of $\ell$, and the difference in weights of mixed type
  vertex is a multiple of $\gcd(m,\ell)$.  See Remark~\ref{cyc-red}. 

  In what follows, only the property of being a potential disk is used
  and therefore this ambiguity in the weighed turn graph associated to
  a conjugacy class is not an issue. 
\end{remark}

\begin{lemma}\label{potential-disk}
  Suppose $\gamma$ is a turn circuit for $S_0$ that corresponds to an
  inner boundary component in $\partial^- S_0$ that bounds a disk in
  $S_1$.  Then $\gamma$ is a potential disk.
\end{lemma}

\begin{proof}
For any band in $S_b$, the core arc (a component of $N_b$) maps to $C'$
as a loop of some degree $d$. The two sides then map to $C$ as loops of
degrees $dm$ and $d\ell$ respectively. 

  If $\iota$ is the side of a band that leads from a turn
  labeled $ta^kt^*$ to a turn labeled $t^*a^{k'}t^{-1}$ then the map
  $\iota \to C$ has degree a multiple of $m$. Likewise if
  $\iota$ leads from a turn labeled $t^{-1}a^kt^*$ to a turn labeled
  $t^*a^{k'}t$ then $\iota$ maps to $C$ with degree a multiple of
  $\ell$. 

  Therefore the total degree of an inner boundary component
  corresponding to a turn circuit $\gamma$ is $\omega(\gamma) + dm +
  d'\ell$ for some integers $d,d'$.  If $\gamma$ is of type $m$ then
  $d' = 0$.  Likewise, if $\gamma$ is of type $\ell$ then $d = 0$.

  If the boundary component actually bounds a disk in $S_1$ then this
  total degree is $0$. Hence $\omega(\gamma) \equiv 0 \mod
  \mu(\gamma)$ and therefore $\gamma$ is a potential disk.
\end{proof}

\section{Linear optimization}\label{sec:lp}

We would like to convert the optimization problem in Lemma~\ref{scl-g} to a problem of optimizing a certain linear functional on a vector space whose coordinates correspond to possible potential disks, subject to certain linear constraints.  Here the functional would count the number of potential disks, and the constraints would arise from the pairing of edges in the turn graph.  The objective would then be to compute stable commutator length using classical linear programming.

The main difficulty in such an approach is arranging that the optimization takes place over a \emph{finite dimensional} object.  In this section, we show how to convert an admissible surface to a vector in a finite dimensional vector space in such a way that the number of disk components of $S_1$ is less than the value of an appropriate linear functional.  We thus obtain computable, rational lower bounds for the stable commutator length of elements of Baumslag--Solitar groups (Theorem~\ref{scl-lowerbound-lp}).  In Section~\ref{sec:alt}, we will show that these bounds are sharp for a certain class of elements. 

We construct the finite dimensional vector space as follows.  Let $w$ be a conjugate of $g$ of the form given in Remark~\ref{cyc-red}.   Let $M=\max\{\abs{m},\abs{\ell}\}$.  We consider two sets of directed
circuits in $\Gamma(w)$:
\begin{itemize}
\item $\frakX$: the set of potential disks that are a sum of not more
  than $M$ embedded circuits, and
\item $\frakY$: the set of all embedded circuits.
\end{itemize}
Note that both $\frakX$ and $\frakY$ are finite sets and that they
may have some circuits in common.  Enumerate these sets as $\frakX = \{
\alpha_1,\ldots,\alpha_p \}$ and $\frakY = \{\beta_1, \ldots, \beta_q\}$.
Let $\X$ be a $p$--dimensional real vector space with basis
$\{\bx_1,\ldots,\bx_p\}$, and let $\Y$ be a $q$--dimensional real vector space
with basis $\{\by_1,\ldots,\by_q\}$.  Equip both $\X$ and $\Y$ with an
inner product that makes the respective bases orthonormal.  By
Remark~\ref{turngraph}, the vector spaces $\X$ and $\Y$ depend only on
the conjugacy class in $BS(m,\ell)$ represented by $w$.  Abusing notation, we let $\{\bx_1,\ldots,\bx_p,\by_1,\ldots,\by_q\}$ denote the corresponding orthonormal basis of $\X\oplus\Y$.  This is the vector space with which we will work.

The linear functional on the vector space $\X\oplus\Y$ whose values will be compared with the number of disk components of $S_1$ is the functional that is the sum of the coordinates corresponding to $\X$, i.e.\ the functional that takes in $\bu
\in \X \oplus \Y$ and gives out $\abs{\bu}_\X := \sum_{i = 1}^p \bu\cdot\bx_i$.  One thinks of this functional as counting the number of potential disks.

There are additional linear functionals on $\X\oplus\Y$ that count the number of times turn circuits in a given collection visit a specific vertex or edge.  For each vertex $v \in \Gamma(w)$, define $F_v\co \X \oplus \Y \to \R$ by letting $F_v(\bx_i)$ be the number of times $\alpha_i$ visits $v$, letting $F_v(\by_i)$ be the number of times $\beta_i$ visits $v$, and extending by linearity.  For each edge $e \subset \Gamma(w)$, define $F_e\co \X \oplus \Y \to \R$  by letting $F_e(\bx_i)$ be the number of times $\alpha_i$ traverses $e$, letting $F_v(\by_i)$ be the number of times $\beta_i$ traverses $e$, and extending by linearity. 

One thinks of the next lemma as saying that, if a collection of turn circuits traverses each edge the same number of times as its dual edge, then this collection of turn circuits visits each vertex the same number of times.

\begin{lemma}\label{eqaul-v}
  If $F_e(\bu) = F_{\bar{e}}(\bu)$ for each dual edge pair
$e,\bar{e}$ of $\Gamma(w)$, then  $F_v(\bu) = F_{v'}(\bu)$ for any vertices $v,v' \in \Gamma(w)$.
\end{lemma}

\begin{proof}
  For a vertex $v \in \Gamma(w)$, let $E^+(v)$ be the set of
  directed edges that are outgoing from  $v$ and let $E^-(v)$ be the set 
  directed edges that are incoming to $v$.  Then
  \begin{equation*}
    F_v(\bu) = \sum_{e \in E^+(v)} F_e(\bu) = \sum_{e \in E^-(v)}
    F_e(\bu).
  \end{equation*}
  First suppose that $v$ corresponds to turn $i$ and $v'$ corresponds
  to turn $i-1$.  Then edge duality gives a pairing between edges in $E^+(v)$ and edges in
  $E^-(v')$.  Since $F_e(\bu) = F_{\bar{e}}(\bu)$ for every dual edge pair $e,\bar{e}$, we have that
  \begin{equation*}
    F_v(\bu) = \sum_{e \in E^+(v)} F_e(\bu) = \sum_{e \in
      E^+(v)} F_{\bar{e}}(\bu) =\sum_{e \in
      E^-(v')} F_{e}(\bu) = F_{v'}(\bu).
  \end{equation*}
  Letting $i$ vary, we obtain a similar statement for all pairs of vertices corresponding to adjacent turns.  It follows that $F_v(\bu) = F_{v'}(\bu)$ for any vertices $v,v' \in \Gamma(w)$.
\end{proof}

Let $C \subset \X \oplus \Y$ be the cone of non-negative vectors $\bu$
such that $F_e(\bu) = F_{\bar{e}}(\bu)$ for every dual edge pair
$e,\bar{e}$ of $\Gamma(w)$.  
In light of the lemma, we denote by $F \co C \to \R$ the function
$F_v\big|_C$ for any vertex $v \in \Gamma(w)$.  

The following proposition shows how to convert an admissible surface into a vector $\bu\in C$ in such a way that $\abs{\bu}_\X$ is at least the number of disk components of $S_1$.

\begin{proposition}\label{better-vector}
  Given an admissible surface $S \to X$, there is a vector $\bu \in C$
  such that
  \begin{equation}\label{eq:better-vector} 
    \frac{\abs{\bu}_\X}{F(\bu)} \geq \frac{\ndisks(S_1)}{n(S)}. 
  \end{equation}
\end{proposition}

\begin{proof}
Suppose the surface $S$ has been decomposed as $S_0\cup S_1$ as described in Section~\ref{sec:turn}, and consider the collection of turn circuits for $S_0$.  Let $\gamma$ be one of the turn circuits in this collection.  As a cycle we can decompose $\gamma$ as a sum of embedded circuits, i.e.\  $\gamma = \beta_{i_1} + \cdots + \beta_{i_k}$.  This decomposition may not be unique, but we only need its existence.  

For each turn circuit $\gamma$, we will construct a corresponding vector $\bu(\gamma)$, depending on this decomposition and on whether the corresponding boundary component of $\partial^- S_0$ bounds a disk in $S_1$.
  If the corresponding inner boundary component of
  $\partial^- S_0$ does not bound a disk in $S_1$, we
  define \[\bu(\gamma) = \sum_{j = 1}^k \by_{i_j}.\] 
Otherwise,
  the corresponding inner boundary component of $\partial^- S_0$ does
  bound a disk in $S_1$, in which case Lemma~\ref{potential-disk} implies that
  $\gamma$ is a potential disk.  If $k \leq M$, then $\gamma  \in\frakX$; say $\gamma=
  \alpha_i$.  In this case we define $\bu(\gamma)
  = \bx_i$.  Otherwise, if $k>M$, we proceed as follows.  For each $\beta_{i_j}$, let
  $\mu(\beta_{i_j})\beta_{i_j}$ denote the sum of $\mu(\beta_{i_j})$ copies of $\beta_{i_j}$.  Notice that $\mu(\beta_{i_j})\beta_{i_j}$ is a potential disk that is not the
  sum of more than $M$ embedded circuits.  Hence
  $\mu(\beta_{i_j})\beta_{i_j}\in\frakX$, so  $\mu(\beta_{i_j})\beta_{i_j}= \alpha_{i'_j}$ for some $i'_j \in
  \{1,\ldots,p\}$.  In this case we define \[\bu(\gamma) = \sum_{j =
    1}^k \frac{1}{\mu(\beta_{i_j})} \bx_{i'_j}.\] 

The vector we will consider is  $\bu
  = \sum_\gamma \bu(\gamma)$, where this sum is taken over all $\gamma$ in the collection of turn circuits for $S_0$ (with multiplicity).  Establishing the following three claims will complete
  the proof of the proposition.
  \begin{enumerate}
  \item $\bu \in C$,\label{u-in-C}
  \item $F(\bu) = n(S)$,\label{F-is-n} and
  \item $|\bu|_\X \geq \ndisks(S_1)$.\label{geq-disk}
  \end{enumerate}

  \eqref{u-in-C}:  The vector $\bu(\gamma)$ was constructed so that $F_e(\bu(\gamma))$ counts the number of times the turn circuit $\gamma$ traverses the edge $e$.  
 Thus $F_e(\bu)$ records the number of times  turn circuits for $S_0$ traverse $e$.  Every time an edge $e$ is traversed by a turn circuit for $S_0$, there is a band in $S_b$ one side of which represents $e$.  The other side of this band represents $\bar{e}$, so therefore we have that $F_e(\bu) = F_{\bar{e}}(\bu)$ for all edges $e$.  Thus $\bu\in C$.

  \eqref{F-is-n}: The vector $\bu(\gamma)$ was also constructed so that $F_v(\bu(\gamma))$ counts the number of times the turn circuit $\gamma$ visits the vertex $v$.  Therefore 
  $F(\bu)$ records the number of times  turn circuits for $S_0$
  visit any given vertex.  As each turn occurs once in $w$, each vertex must be visited exactly
   $n(S)$ times.  Thus $F(\bu)=n(S)$.

  \eqref{geq-disk}: Let $\gamma$ be a turn circuit for $S_0$, and suppose the corresponding inner boundary component of $S_0$ bounds a disk in $S_1$.   Decompose $\gamma$ as a sum 
  $\beta_{i_1} + \cdots +
  \beta_{i_k}$ of embedded circuits as above.   If $k \leq M$, then $\bu(\gamma) =
  \bx_i$ for some $i \in \{1,\ldots,p\}$, and thus
  $\abs{\bu(\gamma)}_\X = 1$.  Otherwise, we have
  \[ \abs{\bu(\gamma)}_\X = \sum_{j = 1}^k \frac{1}{\mu(\beta_{i_j})}
  \geq \sum_{j = 1}^k \frac{1}{M} = \frac{k}{M} \geq 1. \] 
In either case,  $\abs{\bu(\gamma)}_\X \geq 1$.  As
  $\abs{\param}_\X$ is a linear functional, we thus have that $\abs{\bu}_\X \geq
  \ndisks(S_1)$.
\end{proof}

\begin{theorem}\label{scl-lowerbound-lp}
  Let $g \in BS(m,\ell)$, $m \neq \ell$,  be of $t$--exponent
  zero.  Then there is a computable, finite sided, rational polyhedron
  $P \subset \X \oplus \Y$ such that
  \begin{equation}\label{eq:scl-lowerbound-lp}
    \scl(g) \geq \frac{|g|_t}{4} - \frac{1}{2}\max\left\{
      \abs{\bu}_\X \mid \bu \text{ is a vertex of } P \right\}.
  \end{equation}
\end{theorem}

\begin{proof}
  Let $P = F^{-1}(1)$.  If $V$ is the number of vertices of $\Gamma(w)$, we can extend $F$ to a linear functional $\widetilde{F}\co \X \oplus \Y \to \R$ by setting
  $\widetilde{F}(\bu) = \frac{1}{V}\sum_v F_v(\bu)$, where the sum is taken 
  over all vertices of $\Gamma(w)$.  The
  linear functional $\widetilde{F}$ is positive on all basis vectors of
  $\X \oplus \Y$, and hence a level set of $\widetilde{F}$ intersects
  the positive cone in a compact set.  Clearly $P = C \cap
  \widetilde{F}^{-1}(1)$. Thus $P$ is a finite sided, rational, compact polyhedron.  

  By Lemma~\ref{scl-g}, Proposition~\ref{better-vector}, and the
  linearity of $\abs{\param}_\X$ and $F$, we have that
  \begin{align*}
    \scl(g) & = \ \frac{\abs{g}_t}{4} \ + \
     \frac{1}{2}\inf_S \frac{-\ndisks(S_1)}{n(S)} \\
&= \ \frac{\abs{g}_t}{4} \ - \
    \frac{1}{2}\sup_S \frac{\ndisks(S_1)}{n(S)} \\
    & \geq \ \frac{\abs{g}_t}{4} \ - \
    \frac{1}{2}\sup_{\bu \in C} \frac{\abs{\bu}_\X}{F(\bu)} \\
    & \geq \ \frac{\abs{g}_t}{4} \ - \ \frac{1}{2}\sup_{\bu \in P}
    \abs{\bu}_\X.
  \end{align*}
  As $P$ is a finite sided, compact polyhedron and $\abs{\param}_\X$
  is a linear functional, the supremum is realized at one of the
  vertices of $P$.  This gives \eqref{eq:scl-lowerbound-lp}.
\end{proof}

\begin{remark}\label{only-X}
  If $\bu$ is a vertex of $P$ that maximizes $\abs{\bu}_\X$ in $P$,
  then $\bu\cdot\by_i = 0$ for all $i \in \{1,\ldots,q\}$.  Indeed,
  suppose not and let $\beta_i \in \frakY$ be such that $\bu\cdot\by_i  = c > 0$.  Then there is some $\alpha_{i'} \in \frakX$ such that
  $\mu(\beta_i)\beta_i = \alpha_{i'}$.  One then observes that $\bu' =
  \bu - c\by_i + \frac{c}{\mu(\beta_i)}\bx_{i'} \in P$ and
  $\abs{\bu'}_\X > \abs{\bu}_\X$.
\end{remark}

The linear programming problem described in this section has been
implemented using \verb=Sage=~\cite{sage} and is available from the
first author's webpage.  The number of embedded circuits in $\Gamma(w)$
is on the order of $\abs{w}_{t}!$ and so the algorithm is only practical
for elements with small $t$--length. 

\section{Elements of alternating $t$--shape}\label{sec:alt}

The bounds given in Theorem~\ref{scl-lowerbound-lp} are not always sharp,
as we will point out in Remark~\ref{strict inequality}.  However, we show
in Theorem~\ref{scl-equality-alternating} that these bounds are sharp for
a class of elements that have what we call \emph{alternating $t$--shape}.
We thus show that stable commutator length is computable and rational for
such elements.  We also characterize which elements of alternating
$t$--shape admit extremal surfaces
(Theorem~\ref{extremal-characterization}).

\begin{definition}
  We say that an element $g\in BS(m,\ell)$ has \emph{alternating
    $t$--shape} if it has a conjugate of the form given in Remark~\ref{cyc-red}
   where $n$ is even and $\epsilon_i =
  (-1)^{i-1}$.
\end{definition}

In this section, we restrict attention to elements of alternating
$t$--shape and express this conjugate as
\begin{equation}\label{eq:alternating-form}
  w=\prod_{k=1}^r ta^{i_k} t^{-1}a^{j_k}. 
\end{equation}
Note that if $g$ has alternating $t$--shape then it has $t$--exponent
zero. Hence stable commutator length is finite for elements of
alternating $t$--shape.

\subsection*{Constructing surfaces}

Let $P$ be as in the proof of Theorem~\ref{scl-lowerbound-lp}, and let
\[ L(g)=\frac{|g|_t}{4} - \frac{1}{2}\max\left\{ \abs{\bu}_\X \mid \bu
  \text{ is a vertex of } P \right\}. \] 
To show that the lower bound $L(g)$ on stable commutator length is sharp, we would like to find a surface $S$ that gives the same upper bound on stable commutator length.  Specifically,
given a vertex $\bu\in P$, we want to construct a corresponding surface $S=S_0\cup S_1$ of the type discussed in Section~\ref{sec:turn}, where $\partial S$ maps to loops representing conjugacy classes of powers of $g$ and $\partial S_1$ maps to loops representing conjugacy classes of powers of $a$.  Such a surface $S_0$ can be built (in fact, many such surfaces can be built); the construction is given in the proof of Theorem~\ref{scl-equality-alternating}.  The difficulty is arranging $S_0$ so that its inner boundary components can be efficiently capped off by $S_1$.

If the degree of each inner boundary component of $S_0$ were zero, we could take each component of $S_1$ to be a disk.  In this case, we would have $\abs{\bu}_\X = \ndisks(S_1) = \chi(S_1)$ and 
\[
\scl(g) \leq \frac{-\chi(S)}{2n(S)} = \frac{\abs{g}_t}{4} -
\frac{\chi(S_1)}{2n(S)} = L(g) \leq \scl(g).
\]
This would mean the bound in Theorem~\ref{scl-lowerbound-lp} is sharp and the surface $S$ is extremal.

It may not be the case that all inner boundary components of $S_0$ can be made to have degree zero.  Nevertheless, when $g$ has alternating $t$--shape, we can control the number of inner boundary components of $S_0$ that have nonzero degree in such a way as to show that there are surfaces $S$ for which $\frac{-\chi(S)}{2n(S)}$ is arbitrarily close to $L(g)$.  The details are given in the proof of Theorem~\ref{scl-equality-alternating}.  In this way, we establish that the lower bound given in Theorem~\ref{scl-lowerbound-lp} is sharp for elements of alternating $t$--shape.

\begin{theorem}\label{scl-equality-alternating}
  Let $g \in BS(m,\ell)$, $m\neq\ell$, have alternating $t$--shape.
  Then \[ \scl(g)=L(g). \] 
\end{theorem}

\begin{proof}
We will show that $\scl(g)<L(g)+\epsilon$ for all $\epsilon>0$.
Note that, since $g$ is of alternating $t$--shape, all circuits in the
turn graph are either of type $m$ or of type $\ell$, not of mixed type.
Let $\bu$ be a vertex of $P$ on which $\abs{\param}_\X$ is maximal.
Since $P$ is a rational polyhedron on which all coordinates are
nonnegative, all coordinates of $\bu$ are nonnegative rational numbers.
By Remark~\ref{only-X}, $\bu$ has nonzero entries only in coordinates corresponding to $\X$.
Let $K$ denote the number of edges in the turn graph.  Let $N$ be an
integer such that each coordinate of $N\bu$ is a nonnegative integer and
such that $N>K/2\epsilon$ (so that $K/2N<\epsilon$).

Each coordinate $\bx_i$ of $\X$ represents a directed circuit $\gamma$ in
the turn graph.  For each such directed circuit $\gamma$ of length $n$,
we consider a $2n$--gon with alternate sides labeled by the powers of $a$
corresponding to the vertices of the turn graph through which $\gamma$
passes and alternate sides labeled by the intervening edges of the turn
graph traversed by $\gamma$.  See Figure~\ref{fig:polygon}.

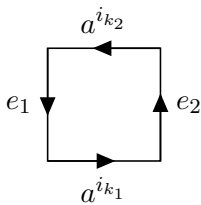
\begin{figure}[ht]
\begin{tikzpicture}[scale=1.5,semithick,>=triangle 45]
\draw (0,0)--(1,0)--(1,1)--(0,1)--cycle;
\draw[->] (0,0)--(0.6,0);
\draw[->] (1,0)--(1,0.6);
\draw[->] (1,1)--(0.4,1);
\draw[->] (0,1)--(0,0.4);
\node[below=2pt] at (0.5,0) {$a^{i_{k_1}}$};
\node[above=2pt] at (0.5,1) {$a^{i_{k_2}}$};
\node[left=2pt] at (0,0.5) {$e_1$};
\node[right=2pt] at (1,0.5) {$e_2$};
\end{tikzpicture}
\caption{A sample polygon corresponding to a circuit $\gamma$ of length
  $2$.}\label{fig:polygon}
\end{figure}

For each $i$ we take $N\bu\cdot\bx_i$ copies of the polygon corresponding
to $\bx_i$, thus obtaining a collection $Q_1,\dotsc,Q_s$ of polygons.
Since $F_e(N\bu)=F_{\bar{e}}(N\bu)$, there exists a pairing of the edges
of these polygons corresponding to edges of the turn graph such that each
edge labeled by $e$ on a polygon $Q_i$ is paired with an edge labeled by
$\bar{e}$ on a polygon $Q_j$. Let $\Delta$ be the graph dual to this
pairing, i.e.\ the graph with a vertex for each polygon $Q_i$ and an edge
between the vertex corresponding to $Q_i$ and the vertex corresponding to
$Q_j$ for each edge of $Q_i$ that is paired with an edge from $Q_j$.  The
graph $\Delta$ may have many components.  However, we can adjust the
pairings of edges of polygons to obtain some control over the number of
components of $\Delta$.  Suppose $Q_{i_1}$ and $Q_{i_2}$ are polygons
where an edge labeled $e$ of $Q_{i_1}$ has been paired with an edge
labeled $\bar{e}$ of $Q_{i_2}$, and suppose $Q_{j_1}$ and $Q_{j_2}$ are
polygons in another component of $\Delta$ where an edge labeled $e$ of
$Q_{j_1}$ has been paired with an edge labeled $\bar{e}$ of $Q_{j_2}$.
Then we can modify the pairing of edges to instead pair the edge labeled
$e$ of $Q_{i_1}$ with the edge labeled $\bar{e}$ of $Q_{j_2}$ and the
edge labeled $e$ of $Q_{j_1}$ with the edge labeled $\bar{e}$ of
$Q_{i_2}$.  The graph $\Delta$ corresponding to this pairing will have
one fewer component than the graph corresponding to the original pairing.  See Figure~\ref{fig:edge-swap}.
Such a modification can be done any time there are two components of
$\Delta$ on which edges with the same labels have been paired.
Therefore, we can arrange that the number of components of $\Delta$ is no
more than $K$, the number of edges in the turn graph.  Note that $\Delta$
is naturally a bipartite graph, with vertices partitioned into those
corresponding to turn circuits of type $m$ (``type $m$ vertices'') and
those corresponding to turn circuits of type $\ell$ (``type $\ell$
vertices'').

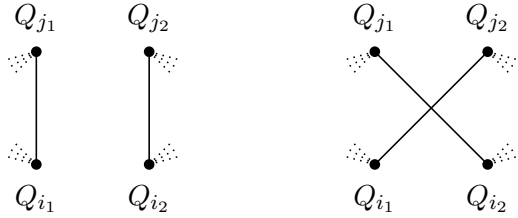
\begin{figure}[ht!]
\begin{tikzpicture}[scale=1.5,semithick,>=triangle 45]
\filldraw (-1,1) circle [radius=0.04cm];
\filldraw (-2,0) circle [radius=0.04cm];
\filldraw (-1,0) circle [radius=0.04cm];
\filldraw (-2,1) circle [radius=0.04cm];
\filldraw (1,0) circle [radius=0.04cm];
\filldraw (1,1) circle [radius=0.04cm];
\filldraw (2,0) circle [radius=0.04cm];
\filldraw (2,1) circle [radius=0.04cm];
\draw[-] (-2,0) -- (-2,1);
\draw[-] (-1,0) -- (-1,1);
\draw[-] (1,0) -- (2,1);
\draw[-] (2,0) -- (1,1);
\node at (-2,-0.3) {$Q_{i_{1}}$};
\node at (-1,-0.3) {$Q_{i_{2}}$};
\node at (1,-0.3) {$Q_{i_{1}}$};
\node at (2,-0.3) {$Q_{i_{2}}$};
\node at (-2,1.3) {$Q_{j_{1}}$};
\node at (-1,1.3) {$Q_{j_{2}}$};
\node at (1,1.3) {$Q_{j_{1}}$};
\node at (2,1.3) {$Q_{j_{2}}$};
\foreach \x in {1,...,3}
{
	\draw [dotted] (-2,0) -- ({-2+0.3*cos(120+15*\x)},{0.3*sin(120+15*\x)});
	\draw [dotted] (-2,1) -- ({-2+0.3*cos(120+15*\x)},{1-0.3*sin(120+15*\x)});
	\draw [dotted] (-1,0) -- ({-1+0.3*cos(15*\x)},{0.3*sin(15*\x)});
	\draw [dotted] (-1,1) -- ({-1+0.3*cos(15*\x)},{1-0.3*sin(15*\x)});
	\draw [dotted] (2,0) -- ({2-0.3*cos(120+15*\x)},{0.3*sin(120+15*\x)});
	\draw [dotted] (2,1) -- ({2-0.3*cos(120+15*\x)},{1-0.3*sin(120+15*\x)});
	\draw [dotted] (1,0) -- ({1-0.3*cos(15*\x)},{0.3*sin(15*\x)});
	\draw [dotted] (1,1) -- ({1-0.3*cos(15*\x)},{1-0.3*sin(15*\x)});
	}
\end{tikzpicture}
\caption{Modifying the edges in $\Delta$ to reduce the number of components.}\label{fig:edge-swap}
\end{figure}

If a vertex $v\in\Delta$ corresponds to a turn circuit $\gamma$, we
define the weight of $v$ to be $\omega(v):=\omega(\gamma)$.  If $v$ is a
type $m$ vertex we have that $m\mid\omega(v)$, and if $v$ is a type
$\ell$ vertex we have that $\ell\mid\omega(v)$.  We wish to assign an
integer $\omega(e)$ to each edge $e\in\Delta$ such that, whenever $v$ is
of type $m$, we have
\begin{equation}\label{eq:type-m-weight}
\omega(v)-m\sum_{e\ni v}\omega(e)=0,
\end{equation}
and, whenever $v$ is of type $\ell$, we have
\begin{equation}\label{eq:type-l-weight}
\omega(v)+\ell\sum_{e\ni v}\omega(e)=0.
\end{equation}
On each connected component of $\Delta$, proceed as follows.  For each
type $m$ vertex $v$, choose a preferred edge $e_v$ emanating from $v$.
For each $v$, set $\omega(e_v)=\omega(v)/m$, and let $\omega(e)=0$ for
all other edges $e$.  This makes \eqref{eq:type-m-weight} hold for all
vertices of type $m$.  Now choose a preferred type $\ell$ vertex $v_0$.
For another type $\ell$ vertex $v_1$, let
\[
\Omega=\frac{\omega(v_1)}{\ell}+\sum_{e\ni v_1}\omega(e).
\]
Choose a path $e_1,\dotsc,e_k$ connecting $v_1$ to $v_0$, and modify the
weights $\omega(e_i)$ by decreasing $\omega(e_i)$ by $\Omega$ whenever
$i$ is odd and increasing $\omega(e_i)$ by $\Omega$ whenever $i$ is even.
For all vertices other than $v_1$ and $v_0$, this does not change the
quantities in \eqref{eq:type-m-weight} and \eqref{eq:type-l-weight}.
Moreover, this causes \eqref{eq:type-l-weight} to now be true for $v_1$.
Fixing $v_0$ and letting $v_1$ vary over all type $\ell$ vertices other
than $v_0$, we obtain edge weights $\omega(e)$ such that
\eqref{eq:type-m-weight} and \eqref{eq:type-l-weight} are true for all
vertices on this component of $\Delta$ except for $v_0$.  Thus we obtain
edge weights $\omega(e)$ such that \eqref{eq:type-m-weight} and
\eqref{eq:type-l-weight} are true for all vertices except for one vertex
in each component of $\Delta$.

We now proceed to build a surface.  Rather than building a surface from
the polygons $Q_i$, we use them to build a band surface $S_0$, then
attempt to fill various components of $\partial S_0$ with disks.  For
each pairing of an edge of $Q_i$ with an edge of $Q_j$, insert a
rectangle with sides labeled by $t$, $a^{m\omega(e)}$, $t^{-1}$, and
$a^{-\ell \omega(e)}$, where $\omega(e)$ is the weight assigned to the
corresponding edge of $\Delta$.  See Figure~\ref{fig:building-surface}.

\begin{figure}[ht]
\begin{tikzpicture}[scale=1.5,semithick,>=triangle 45]
\filldraw[gray!25] (0.5,1.5)--(1,1)--(2,1)--(2,3)--(1.5,3.5)--(3.5,3.5)--(3,3)--(3,1)--(4,1)--(4.5,1.5)--(4.5,0)--(0.5,0)--cycle;
\draw (0,0)--(5,0);
\draw (0,2)--(1,1)--(2,1)--(2,3)--(1,4);
\draw (4,4)--(3,3)--(3,1)--(4,1)--(5,2);
\draw (2,1)--(3,1);
\draw (2,3)--(3,3);
\draw [->] (2,1)--(2.6,1);
\draw [->] (3,1)--(3,2.1);
\draw [->] (3,3)--(2.4,3);
\draw [->] (2,3)--(2,1.9);
\draw [->] (1,0)--(1.6,0);
\draw [->] (2,0)--(2.6,0);
\draw [->] (3,0)--(3.6,0);
\draw [->] (1,1)--(1.6,1);
\draw [->] (3,1)--(3.6,1);
\draw[thin] (1,0.1)--(1,-0.1);
\draw[thin] (2,0.1)--(2,-0.1);
\draw[thin] (3,0.1)--(3,-0.1);
\draw[thin] (4,0.1)--(4,-0.1);
\node at (0.9,2.45) {$Q_i$};
\node at (4.1,2.45) {$Q_j$};
\node[below=2pt] at (1.5,0) {$a^{j_{k-1}}$};
\node[below=4pt] at (2.5,0) {$t$};
\node[below=2pt] at (3.5,0) {$a^{i_k}$};
\node[above=2pt] at (1.5,1) {$a^{j_{k-1}}$};
\node[above=2pt] at (3.5,1) {$a^{i_k}$};
\node[left=2pt] at (2,2) {$a^{-\ell \omega(e)}$};
\node[right=2pt] at (3,2) {$a^{m \omega(e)}$};
\node[below=2pt] at (2.5,1) {$t$};
\node[above=2pt] at (2.5,3) {$t^{-1}$};
\end{tikzpicture}
\caption{Using the polygons $Q_i$ to build a band surface $S_0$.  The
  surface $S_0$ is shaded above.}\label{fig:building-surface}
\end{figure}
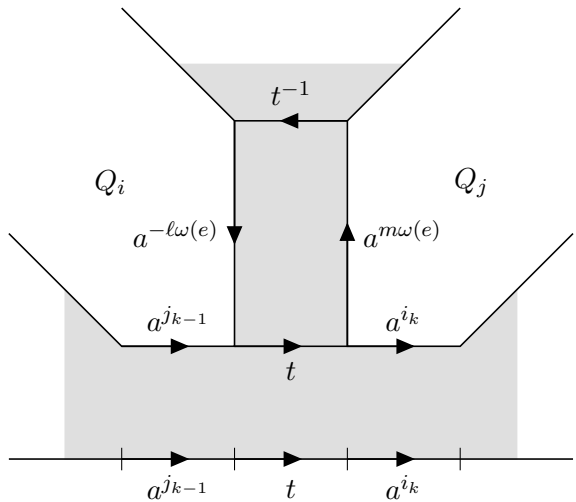

Note that the edges of these rectangles labeled $t$ and $t^{-1}$,
together with the edges of the polygons labeled by powers of $a$, form
paths that would map to the $1$--skeleton of $X$.  By construction, these
paths correspond exactly to powers of $w$.  To each of these paths,
attach an annulus $S^1\times[0,1]$ labeled on both sides by this power of
$w$.  The rectangles and annuli together form $S_0$, shown in
Figure~\ref{fig:building-surface}.  Note that each rectangle maps
naturally to the $2$--cell of $X$ with degree $\omega(e)$.  The annuli
map to the $1$--skeleton of $X$, as indicated by the labels, with the map
factoring through the projection $S^1\times[0,1]\to S^1$.

As in Section~\ref{sec:turn}, we refer to the boundary components of
$S_0$ that map to a power of $w$ as outer boundary components and to
those corresponding to a polygon $Q_i$ as inner boundary components.
Each of the inner boundary components of $S_0$ maps to a power of $a$;
let $d$ be the number of components of the inner boundary for which
this power is zero. The powers of $a$ on the inner boundary components
are exactly the quantities on the left-hand sides of
\eqref{eq:type-m-weight} and \eqref{eq:type-l-weight}. The weights
$\omega(e)$ have been chosen so that these quantities are zero for all
but $K$ components of the inner boundary, so therefore $d\geq s-K$.
Fill these $d$ components of the inner boundary with disks, and call
the resulting surface $S$.

Each inner boundary component of $S$ maps to a power of $a$, and Lemma~\ref{scl-a} says that $\scl(a)=0$.  Therefore, applying Lemma~\ref{fill-with-zero}, we have that

\begin{align*}
  \scl(g)&\leq\frac{-\chi(S)}{2n(S)}\\
  &=\frac{\abs{g}_t}{4}-\frac{d}{2N}\\
  &\leq \frac{\abs{g}_t}{4}-\frac{s-K}{2N}\\
  &= \frac{\abs{g}_t}{4}-\frac{s}{2N}+\frac{K}{2N}\\
  &< \frac{\abs{g}_t}{4}- \frac{1}{2}\max\left\{
    \abs{\bu}_\X \mid \bu \text{ is a vertex of } P
  \right\}+\epsilon\\
&= L(g)+\epsilon.
\end{align*}
Thus $\scl(g)=L(g)$, as desired.
\end{proof}

\begin{corollary}\label{scl-alternating-rational} 
  If $g \in BS(m,\ell)$, $m\neq\ell$, has alternating $t$--shape, then
  $\scl(g)$ is rational. 
\end{corollary}

\begin{proof} 
  Since each vertex $\bu$ of $P$ has rational coordinates and
  $\abs{\bu}_{\X}$ is the sum of certain of these coordinates, we know
  that $L(g)$ is rational. Therefore it follows from
  Theorem~\ref{scl-equality-alternating} that $\scl(g)$ is rational.
\end{proof}

\begin{remark}\label{strict inequality}
  In general one suspects the inequality in
  Theorem~\ref{scl-lowerbound-lp} is strict.  For example, the
  function $\abs{\bu}_\X$ has a unique maximum on the polyhedron $P$
  in Theorem~\ref{scl-lowerbound-lp} for the element
  $a^2t^2at^{-1}at^{-1} \in BS(2,3)$.  When attempting to build a
  surface from this unique optimal vertex $\bu$, it turns out that every
  component of 
  the dual graph $\Delta$ has a constant proportion of vertices that
  cannot be filled, regardless of the edge weights.  

  Therefore, in contrast to Theorem~\ref{scl-equality-alternating},
  where all but at most a constant number of vertices can be filled,
  there is no sequence of surfaces associated to $\bu$ to for
  which $\frac{\ndisks(S)}{n(S)}$ approaches $\abs{\bu}_\X$.  
\end{remark}

\subsection*{An explicit formula} 
When $\abs{g}_t = 2$ the turn graph consists of two vertices, each
adjacent to a one-edge loop. In this case the vector space $\X$ is
essentially two dimensional and the linear optimization problem can be
solved easily by hand, resulting in a formula for stable commutator
length for such elements. 

This calculation is interesting for two
reasons. First, it is rare that one 
can derive a formula for $\scl$ in non-trivial cases. Second, the
minimal value for $\scl$ among all ``well-aligned'' elements (see
Definition~\ref{def:well-aligned} and Theorem~\ref{th:well-aligned})
is realized by an element of this type. 

\begin{proposition}\label{prop:length2}
In the group $BS(m,\ell)$ with $m\not= \ell$, if $m\nmid i$ and $\ell
\nmid j$ then 
\[
\scl(ta^i t^{-1}a^j) \ = \ 
\frac{1}{2} \left( 1 - \frac{\gcd(i,m)}{\abs{m}} -
  \frac{\gcd(j,\ell)}{\abs{\ell}}\right).  
\]
\end{proposition}
The divisibility hypotheses simply mean that the word $ta^i t^{-1}a^j$ is
cyclically reduced (cf.\ Remark \ref{cyc-red}). 

\begin{proof}
The turn graph for the word $ta^it^{-1}a^j$ is as shown in
Figure~\ref{fig:turn-example}.

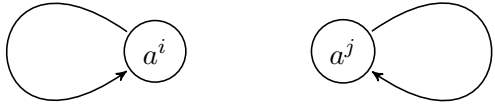
\begin{figure}[ht!]
\begin{tikzpicture}[semithick, >=stealth', shorten <=1pt, shorten >=1pt,scale=2.5]
 \clip (-0.8,-0.28) rectangle (1.8,0.28);  
\node [circle,draw] at (0,0) (left) {$a^i$};
\node [circle,draw] at (1,0) (right) {$a^j$};
\draw[->] (left) .. controls (-1,0.7) and (-1,-0.7) ..  (left);
\draw[->] (right) .. controls (2,0.7) and (2,-0.7) .. (right);
\end{tikzpicture}
\caption{The turn graph for $ta^it^{-1}a^j$.}\label{fig:turn-example}
\end{figure}

There are two types of potential disks:
\begin{enumerate}
\item Circuits of type $m$ that traverse the left loop of the turn graph
  $p$ times, where $m\mid pi$.
\item Circuits of type $\ell$ that traverse the right loop of the turn
  graph $q$ times, where $\ell\mid qj$.
\end{enumerate}
Note that the condition $m\mid pi$ is equivalent to
$\frac{\abs{m}}{\gcd(i,m)}\mid p$, and the condition $\ell\mid qj$ is
equivalent to $\frac{\abs{\ell}}{\gcd(j,\ell)}\mid q$.  Suppose
$p=\frac{k\abs{m}}{\gcd(i,m)}$, where $p\leq\max\{\abs{m},\abs{\ell}\}$,
for some positive integer $k$, and let $\bx_{i_k}$ be the corresponding
basis vector of $\X$.  We claim that, if $k>1$ and $\bu$ is a vertex of
$P$ that maximizes $\abs{\bu}_{\X}$, then $\bu\cdot\bx_{i_k}=0$.  Indeed,
suppose not, and consider the vector
$\bu'=\bu-(\bu\cdot\bx_{i_k})\bx_{i_k}+k(\bu\cdot\bx_{i_k})\bx_{i_1}$.
Then $F_e(\bu')=F_e(\bu)$ for all $e$ and $F(\bu')=F(\bu)=1$, but
\[
\abs{\bu'}_{\X}=\abs{\bu}_{\X}-\bu\cdot\bx_{i_k}+k(\bu\cdot\bx_{i_k})=\abs{\bu}_{\X}+(k-1)(\bu\cdot\bx_{i_k})>\abs{\bu}_{\X}. 
\]
A similar argument applies to coordinates of $\X$ corresponding to
potential disks of type $\ell$.  Thus, if $\bu$ is a vertex of $P$ that
maximizes $\abs{\bu}_{\X}$, only two coordinates of $\bu$ are nonzero,
one corresponding to a potential disk of type $m$ where
$p=\frac{\abs{m}}{\gcd(i,m)}$ and the other corresponding to a potential
disk of type $\ell$ where $q=\frac{\abs{\ell}}{\gcd(j,\ell)}$.

Let $c$ be the value of the coordinate corresponding to this potential
disk of type $m$, and let $d$ be the value of the coordinate
corresponding to this potential disk of type $\ell$.  Then the conditions
$F_e(\bu)=F_{\bar{e}}(\bu)$ and $F(\bu)=1$ become
\[
\frac{\abs{m}}{\gcd(i,m)}c=\frac{\abs{\ell}}{\gcd(j,\ell)}d=1.
\]
Therefore we have that $c=\frac{\gcd(i,m)}{\abs{m}}$ and
$d=\frac{\gcd(j,\ell)}{\abs{\ell}}$.  This means that 
\begin{align*}
\scl(ta^it^{-1}a^j)&=\frac{|ta^it^{-1}a^j|_t}{4} - \frac{1}{2}\max\left\{
      \abs{\bu}_\X \mid \bu \text{ is a vertex of } P \right\}\\
&=\frac12-\frac12\left(\frac{\gcd(i,m)}{\abs{m}}+\frac{\gcd(j,\ell)}{\abs{\ell}}\right), 
\end{align*}
as desired. 
\end{proof}

\subsection*{Extremal surfaces}
We now characterize the elements $g\in BS(m,\ell)$ of alternating
$t$--shape for which an extremal surface exists.
\begin{lemma}\label{extremal-alternating-disks}
  Suppose $S$ is an admissible surface for some $g\in BS(m,\ell)$,
  $m\neq\ell$, of alternating $t$--shape that has been decomposed as
  described in Section~\ref{sec:turn}.  If $S$ is extremal, then $S_1$
  consists only of disks.
\end{lemma}
\begin{proof}
  If $S$ is extremal, we know by Lemma~\ref{extremal-disks-annuli} that
  $S_1$ consists of only disks and annuli.  Suppose that some component
  of $S_1$ is an annulus.  This means that some component of the inner
  boundary of $S_0$ does not bound a disk in $S$.  Using the construction
  from the proof of Proposition \ref{better-vector}, there is a $\bu \in P$ such
  that $\abs{\bu}_\X \geq \ndisks(S_1)$ and $\bu\cdot\by_i>0$ for some $i$. 
Remark~\ref{only-X} shows how to find
  $\bu'\in P$ such that $\abs{\bu'}_{\X}>\abs{\bu}_{\X}$, so we have
  $\abs{\bu'}_{\X}>\ndisks(S_1)$.  But then
  Theorem~\ref{scl-equality-alternating} shows that $S$ is not extremal.
  Thus $S_1$ cannot have an annular component, meaning it consists only
  of disks.
\end{proof}

\begin{theorem}\label{extremal-characterization} 
  Let $g =\prod_{k=1}^r ta^{i_k} t^{-1}a^{j_k}\in BS(m,\ell)$,
  $m\neq\ell$. There is an extremal surface for $g$ if and only if
  \begin{equation}\label{eq:balance} 
    \ell\sum_{k=1}^r i_k =-m\sum_{k=1}^r j_k. 
  \end{equation} 
\end{theorem}

\begin{proof}
The status of equation \eqref{eq:balance} does not change under
cancellation of $t^{\epsilon} t^{-\epsilon}$ pairs in $g$, nor under
applications of the defining relator in $BS(m,\ell)$; hence we may
assume without loss of generality that $g$ is cyclically reduced. 

  First, suppose $g$ has an
  extremal surface $S$.  Decompose $S$ as described in
  Section~\ref{sec:turn}.  By Lemma~\ref{extremal-alternating-disks},
  $S_1$ consists only of disks.  Let $\Delta$ be the graph that has a
  vertex for each component of $S_1$ and an edge for each band of $S_b$
  that connects the vertices corresponding to the two disks it adjoins.
  There is a weight function on the vertices of $\Delta$, where
  $\omega(v)$ is the total degree of $a$ at all vertices of the circuit
  in the turn graph corresponding to $v$.  There is also a natural weight
  function $w$ on the edges of $\Delta$, where $\omega(e)$ is the signed
  degree of the map from the band corresponding to $e$ to the $2$--cell
  of $X$.  Since all vertices of $\Delta$ bound disks, we know that
  whenever $v$ corresponds to a circuit of type $m$, we have
\[
\omega(v)-m\sum_{e\ni v}\omega(e)=0,
\]
and, whenever $v$ corresponds to a circuit of type $\ell$, we have
\[
\omega(v)+\ell\sum_{e\ni v}\omega(e)=0.
\]
Summing over all vertices of type $m$, we obtain
\begin{equation}\label{eq:type-m-sum}
m\sum_{e\in\Delta}\omega(e)=\sum_{\substack{v\in\Delta\\\text{of type $m$}}}\omega(v)=n(S)\sum_{k=1}^r i_k.
\end{equation}
Summing over all vertices of type $\ell$, we obtain
\begin{equation}\label{eq:type-l-sum}
-\ell\sum_{e\in\Delta}\omega(e)=\sum_{\substack{v\in\Delta\\\text{of type $\ell$}}}\omega(v)=n(S)\sum_{k=1}^r j_k.
\end{equation}
Multiplying \eqref{eq:type-m-sum} by $\ell$ and \eqref{eq:type-l-sum} by
$-m$ and combining gives \eqref{eq:balance}.

Conversely, suppose the element $g$ satisfies \eqref{eq:balance}.  Let $S_0$
and $\Delta$ be as in the proof of
Theorem~\ref{scl-equality-alternating}.  Restrict to one connected
component of $S_0$, and let $\Delta_0$ be the corresponding connected
component of $\Delta$.  Then \eqref{eq:type-m-weight} holds for all
vertices of type $m$ in $\Delta_0$.  Let $N_0$ be the power of $w$
corresponding to the image of the map on the outer boundary of this
component of $S_0$.  Summing over all vertices of type $m$ in $\Delta_0$,
we have that
\begin{equation}\label{eq:type-m-sum0}
m\sum_{e\in\Delta_0}\omega(e)=\sum_{\substack{v\in\Delta_0\\\text{of type $m$}}}\omega(v)=N_0\sum_{k=1}^r i_k
\end{equation}
also holds.  Multiplying \eqref{eq:type-m-sum0} by $\ell$ and combining
with \eqref{eq:balance} shows that
\begin{equation}\label{eq:type-l-sum0}
-\ell\sum_{e\in\Delta_0}\omega(e)=\sum_{\substack{v\in\Delta_0\\\text{of type $\ell$}}}\omega(v)=N_0\sum_{k=1}^r j_k.
\end{equation}
Since the procedure in the proof of
Theorem~\ref{scl-equality-alternating} ensures that
\eqref{eq:type-l-weight} holds for all but one $v\in\Delta_0$ of type
$\ell$, \eqref{eq:type-l-sum0} implies that \eqref{eq:type-l-weight} in
fact holds for all $v\in\Delta_0$.  The same argument applies to each
component of $\Delta$, so hence \eqref{eq:type-m-weight} and
\eqref{eq:type-l-weight} hold for all $v\in\Delta$.  Thus all inner
boundary components of $S_0$ can be filled with disks, meaning the
resulting surface achieves the lower bound on $\scl(g)$ given by linear
programming.  Hence this surface is extremal.
\end{proof}

\begin{remark}
  Corollary~\ref{scl-alternating-rational} and
  Theorem~\ref{extremal-characterization} combine to give many examples
  of elements for which stable commutator length is rational but for
  which no extremal surface exists.  Previous examples of this phenomenon
  were found in free products of abelian groups of higher rank.
  See~\cite{Calegari:sails}.
\end{remark}

\section{Quasimorphisms on groups acting on trees}\label{sec:qm}

We now turn our attention from analyzing $\scl$ for a single element
in $BS(m,\ell)$ to analyzing properties of the $\scl$
spectrum $\scl(BS(m,\ell)) \subset \R$.  Our main theorem about the
spectrum (Theorem~\ref{th:gap}) shows that either $\scl(g) = 0$ or
$\scl(g) \geq 1/12$.  In other words, there is a \emph{gap} in the
spectrum.  The proof has two parts.  In this section we will provide a
general condition (\emph{well-aligned}) for an element $g$ in a group
$G$ acting on a tree that implies $\scl(g) \geq 1/12$
(Theorem~\ref{th:well-aligned}).  This is not quite enough for the Gap
Theorem for Baumslag--Solitar groups.  In Section~\ref{sec:gap} we use
the specific structure of $BS(m,\ell)$ as an HNN extension and show
that if a stronger form of the well-aligned property does not hold, then
$\scl(g) = 0$ (Theorem~\ref{infiniteoverlap} and
Proposition~\ref{sclzero}).

The key to the argument in Theorem~\ref{th:well-aligned} is the
construction of a certain function $f \co BS(m,\ell) \to \R$ for each
hyperbolic element $g \in BS(m,\ell)$, satisfying certain properties
described below, that provides a lower bound on $\scl(g)$.

The material in this section applies to any group $G$.

\subsection*{Quasimorphisms and stable commutator length}

The functions we will construct are homogeneous quasimorphisms.

\begin{definition}\label{def:qm}
  A function $f \co G \to \R$ is called a \emph{quasimorphism} if
  there is a number $D$ such that
\begin{equation}\label{qm}
\abs{f(gh) - f(g) - f(h)} \ \leq \ D
\end{equation}
for all $g, h \in G$. The smallest such $D$ is called the \emph{defect}
of $f$. A quasimorphism $f$ is \emph{homogeneous} if $f(g^n) = n f(g)$
for all $g\in G$, $n \in \Z$. 
\end{definition}

Bavard Duality \cite{Bavard:duality} provides the link between
homogeneous quasimorphisms and stable commutator length.  We only need
one direction of this link.

\begin{proposition}[Bavard Duality, easy direction]\label{bavard}
Given $g\in G$, suppose there is a homogeneous quasimorphism $f$ with
defect at most $D$ such that $f(g) = 1$. Then $\scl(g) \geq 1/2D$. 
\end{proposition}

\begin{proof}
One checks easily using \eqref{qm} that if $g^n$ is a product of $m$
commutators then $\abs{f(g^n)} \leq 2mD$. Hence $1 = \abs{f(g)} \leq 2
\cl(g^n)D/n$, and taking the limit as $n \to \infty$ gives the desired result. 
\end{proof}

Hence to derive a large lower bound for $\scl(g)$ one tries to
construct a homogenous quasimorphism $f\co G \to \R$ such that $f(g) =
1$ and with defect as small as possible.

\subsection*{$G$--trees}
We consider simplicial trees not just as combinatorial objects, but
also as metric spaces with each edge being isometric to an interval of
length one.  A \emph{segment} in a tree $T$ is a subset $\alpha
\subset T$ that is isometric to a closed segment in $\R$. 

Suppose we are given an action of $G$ on a simplicial tree $T$, always
assumed to be without inversions.  Every element $g \in G$ has a
\emph{characteristic subtree} $T_g$, consisting of those points $x \in
T$ where the displacement function $x \mapsto d(x, gx)$ achieves its
minimum. This minimum is denoted $\abs{g}$, and called the
\emph{translation length} of $g$, or simply the \emph{length} of
$g$. If the length is zero then $T_g$ is the set of fixed points of
$g$ and we call $g$ \emph{elliptic}. Otherwise, $T_g$ is a linear
subtree on which $g$ acts by a shift of amplitude $\abs{g}$. In this
case $T_g$ is called the \emph{axis} of $g$ and $g$ is
\emph{hyperbolic}. Note that $T_g$ has a natural orientation, given by
the direction of the shift by $g$. A \emph{fundamental domain} for $g$
is a segment (of length $\abs{g}$) contained in the axis, of the form
$[x,gx]$. We specifically allow $x$ to be a point in the interior of
an edge.

If $k\not=0$ then $g^k$ has the same type (elliptic or hyperbolic) as
$g$. If $g$ is hyperbolic then $T_{g^k} = T_g$ and $\abs{g^k} = \abs{k}
\abs{g}$. Also, $\abs{hgh^{-1}} = \abs{g}$ for all $g, h$. 

\begin{remark}\label{easyway}
  There is an easy way to identify the axis of a hyperbolic element
  $g\in G$. Namely, if $\alpha$ is an oriented segment or edge in $T$,
  then $\alpha$ is on the axis if and only if $\alpha$ and $g\alpha$ are
  \emph{coherently oriented} in $T$, i.e., there is an oriented segment that contains both $\alpha$ and $g\alpha$ as oriented subsegments.  When this occurs, if $x$ is any
  point in $\alpha$, then the segment $[x,gx]$ is a fundamental
  domain for $g$. 
\end{remark}

\begin{definition}
Let $\gamma$ be an oriented segment in $T$. The
\emph{reverse} of $\gamma$ is the same segment with the opposite
orientation, denoted $\overline{\gamma}$. A \emph{copy of
  $\gamma$} is a segment of the form $g \gamma$ for some $g\in G$. 

If $g$ is hyperbolic then the quotient of $T_g$ by the action of
$\langle g \rangle$ is a circuit of length $\abs{g}$. A \emph{copy of
  $\gamma$} in $T_g / \langle g \rangle$ is the image of a copy of
$\gamma$ in $T_g$, provided that $\abs{\gamma} \leq \abs{g}$. (If
$\abs{\gamma} > \abs{g}$ then there are no copies of $\gamma$ in
$T_g/\langle g \rangle$.) We say that two segments \emph{overlap} if
their intersection is a non-trivial segment.
\end{definition}

Let $\gamma$ be an oriented segment in $T$. For an oriented segment
$\alpha$, let $c_{\gamma}(\alpha)$ be the maximal number of
non-overlapping positively oriented copies of $\gamma$ in
$\alpha$. Note that $c_{\gamma}(\overline{\alpha}) =
c_{\overline{\gamma}}(\alpha)$. Also define
\begin{equation*}
f_{\gamma}(\alpha) \ = \ c_{\gamma}(\alpha) -
  c_{\overline{\gamma}}(\alpha).
\end{equation*}

If $g\in G$ is hyperbolic, let $c_{\gamma}(g)$ be the maximal number of
non-overlapping positively oriented copies of $\gamma$ in $T_g/\langle
g \rangle$. If $g$ is elliptic, let $c_{\gamma}(g) = 0$. In either case,
define 
\begin{equation}\label{f-def}
f_{\gamma}(g) \ = \ c_{\gamma}(g) - c_{\overline{\gamma}}(g) 
\end{equation}
and 
\begin{equation}\label{h-def}
 h_{\gamma}(g) \ = \ \lim_{n \to \infty} \frac{f_{\gamma}(g^n)}{n}.
\end{equation}

We will see shortly that $f_{\gamma}$ is a quasimorphism. Therefore, by
\cite[Lemma~2.21]{Calegari:scl}, the limit defining $h_{\gamma}$ exists
and $h_{\gamma}$ is a homogeneous quasimorphism. 

\begin{lemma}\label{junctures}
Let $g\in G$ be hyperbolic and suppose that a fundamental domain for $g$
is expressed as a concatenation of non-overlapping segments $\alpha_1,
\ldots, \alpha_k$, each given the same orientation as $T_g$. Then for any
$\gamma$ there is an estimate 
\[ \sum_i c_{\gamma}(\alpha_i) \ \leq \ c_{\gamma}(g) \ \leq \ k + \sum_i
c_{\gamma}(\alpha_i). \]
\end{lemma}

In the situation of the lemma, we will refer to the images in
$T_g/\langle g \rangle$ of the endpoints of the segments $\alpha_i$ as
\emph{junctures}. There are $k$ junctures in $T_g/\langle g \rangle$.  

\begin{proof}
Start with maximal collections of non-overlapping copies of $\gamma$ in
the segments $\alpha_i$. The union of these sets of copies projects to a
non-overlapping collection in $T_g/\langle g \rangle$, yielding the first
inequality. For the second inequality, start with a maximal
collection of non-overlapping copies of $\gamma$ in $T_g / \langle g
\rangle$. At most $k$ of these copies contain junctures in their 
interiors. Each remaining copy lifts to a copy of $\gamma$ in one of the
segments $\alpha_i$, and no two of these lifts overlap. Hence
$\sum_i c_{\gamma}(\alpha_i) \geq c_{\gamma}(g) - k$. 
\end{proof}

The main technical result of this section is the following theorem.

\begin{theorem}\label{defect}
  Suppose $G$ acts on a simplicial tree $T$. Let $\gamma$ be an
  oriented segment in $T$ (with endpoints possibly not at
  vertices). Then the functions $f_{\gamma}$ and $h_{\gamma}$ defined
  in \eqref{f-def} and \eqref{h-def} are quasimorphisms on $G$ with
  defect at most $6$.
\end{theorem}

\begin{proof}
We will prove the result for $h_{\gamma}$ directly. Replacing ``$n$''
throughout by ``$1$'' yields a proof of the result for $f_{\gamma}$. 

Fix elements $g, h \in G$. We wish to show that $\abs{h_{\gamma}(gh) -
  h_{\gamma}(g) - h_{\gamma}(h)} \leq 6$. There are several cases,
corresponding to different configurations of the characteristic subtrees
$T_g$, $T_h$, and $T_{gh}$. 

\medskip
\emph{Case I: $T_g$ and $T_h$ are disjoint.} Let $\rho$ be the segment
joining $T_h$ to $T_g$, oriented from $T_h$ and towards $T_g$. Let $\rho'
= g\overline{\rho}$ (which is a copy of $\overline{\rho}$). 

If $g$ and $h$ are both hyperbolic, let $\alpha$ and $\beta$ be
fundamental domains for $h$ and $g$ respectively, as indicated in Figure
\ref{fig:I}. 
\begin{figure}[ht]
\labellist
\small\hair 2pt
\pinlabel {$\alpha$} [b] at 146 32
\pinlabel {$\beta$} [t] at 199 12
\pinlabel {$\rho$} [r] at 171 21
\pinlabel {$\rho'$} [r] at 226 23
\pinlabel {$T_h$} [l] at 225 43
\pinlabel {$T_g$} [l] at 278 0
\pinlabel {$T_{gh}$} [l] at 337 27
\endlabellist
\centering
\includegraphics{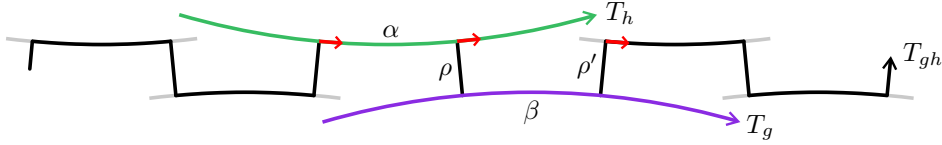}
\caption{Case I, $g$ and $h$ hyperbolic. $T_h$ is green, $T_g$ is purple,
  $T_{gh}$ is black. The red edges are of the form $e$, $he$, and $ghe$.}
\label{fig:I} 
\end{figure}
Note that $gh$ has a fundamental domain given by the concatenation
$\alpha \cdot \rho \cdot \beta \cdot \rho'$, by Remark \ref{easyway}. More
generally, $(gh)^n$ has a fundamental domain made of $n$ copies each of
$\alpha$, $\rho$, $\beta$, and $\overline{\rho}$. Lemma \ref{junctures}
yields 
\begin{align}\label{e1}
n(c_{\gamma}(\alpha) + c_{\gamma}(\rho) + c_{\gamma}(\beta) +
c_{\gamma}(\overline{\rho})) \ &\leq \ c_{\gamma}((gh)^n) \ \notag \\ 
&\leq \ n(c_{\gamma}(\alpha) + c_{\gamma}(\rho) + c_{\gamma}(\beta) + 
c_{\gamma}(\overline{\rho})) + 4n
\end{align}
and
\begin{align}\label{e2}
n(c_{\overline{\gamma}}(\alpha) + c_{\overline{\gamma}}(\rho) +
c_{\overline{\gamma}}(\beta) + c_{\overline{\gamma}}(\overline{\rho})) \
&\leq \ c_{\overline{\gamma}}((gh)^n) \ \notag \\ &\leq \
n(c_{\overline{\gamma}}(\alpha) + c_{\overline{\gamma}}(\rho) +
c_{\overline{\gamma}}(\beta) + c_{\overline{\gamma}}(\overline{\rho})) +
4n. 
\end{align}
Subtracting \eqref{e2} from \eqref{e1} yields
\begin{equation}\label{e3}
n(f_{\gamma}(\alpha) + f_{\gamma}(\beta)) - 4n \ \leq \ f_{\gamma}((gh)^n) \
\leq \ n(f_{\gamma}(\alpha) + f_{\gamma}(\beta)) + 4n. 
\end{equation}
Since $h^n$ and $g^n$ have fundamental domains made of $n$ copies of
$\alpha$ and $\beta$ respectively, Lemma \ref{junctures} also yields, in
a similar way, 
\begin{equation}\label{e4}
nf_{\gamma}(\alpha) - n \ \leq \ f_{\gamma}(h^n) \ \leq \
nf_{\gamma}(\alpha) + n 
\end{equation}
and 
\begin{equation}\label{e5}
nf_{\gamma}(\beta) - n \ \leq \ f_{\gamma}(g^n) \ \leq \
nf_{\gamma}(\beta) + n. 
\end{equation}
Subtracting \eqref{e4} and \eqref{e5} from \eqref{e3} yields
\begin{equation*}
-6n \ \leq \ f_{\gamma}((gh)^n) - f_{\gamma}(g^n) - f_{\gamma}(h^n) \
\leq \ 6n. 
\end{equation*}
Dividing by $n$ and taking a limit, we obtain 
$\abs{h_{\gamma}(gh) - h_{\gamma}(g) - h_{\gamma}(h)} \ \leq \ 6$,
as desired. 

\medskip
\emph{Remark.} In the argument just given, the fundamental
domains for $h$, $g$, and $gh$ respectively were decomposed into one,
one, and four segments; hence $T_h/\langle h\rangle$, $T_g/\langle
g\rangle$, and $T_{gh}/\langle gh \rangle$ contained a total of six
junctures. These junctures were the only source of defect, since the
individual segments always contributed zero to $\abs{f_{\gamma}(gh) -
  f_{\gamma}(g) - f_{\gamma}(h)}$. Every case below follows the same
pattern: the defect 
will be bounded above by the total number of junctures appearing in the
quotient circuits. In what follows, we will describe the structure of 
$T_h/\langle h\rangle$, $T_g/\langle g\rangle$, and $T_{gh}/\langle gh
\rangle$ in each case and leave some of the details of the estimates to
the reader. 

\medskip
Returning to Case I, suppose $g$ and $h$ are both elliptic. Then $g^n$
and $h^n$ are also elliptic, and $(gh)^n$ has a fundamental domain made
of $n$ copies of $\rho$ and $n$ copies of $\overline{\rho}$. See Figure
\ref{fig:I-e}. Using Lemma
\ref{junctures} one obtains 
\[ \abs{f_{\gamma}((gh)^n) - f_{\gamma}(g^n) - f_{\gamma}(h^n)} \ = \
\abs{f_{\gamma}((gh)^n)} \ \leq \ 2n,\]
for a defect of at most $2$. 
\begin{figure}[ht]
\labellist
\small\hair 2pt
\pinlabel {$\rho$} [tr] at 111 24
\pinlabel {$T_h$} [c] at 92 42
\pinlabel {$T_g$} [c] at 129 7
\pinlabel {$T_{gh}$} [l] at 293 26
\endlabellist
\centering
\includegraphics{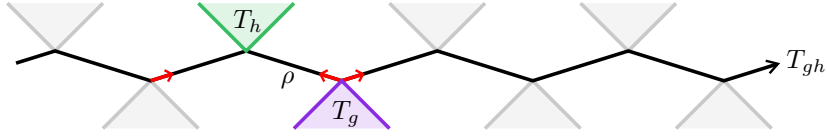}
\caption{Case I, $g$ and $h$ elliptic. $T_h$ is green, $T_g$ is purple,
  $T_{gh}$ is black. The red edges are of the form $e$, $he$, and $ghe$.}
\label{fig:I-e} 
\end{figure}

If one of $g$ and $h$ is elliptic, say $g$, then $h^n$ has a fundamental
domain given by $n$ copies of $\alpha$, and $(gh)^n$ has a fundamental
domain given by $n$ copies each of $\alpha$, $\rho$, and
$\overline{\rho}$. The estimate given by Lemma \ref{junctures} becomes 
\[ \abs{f_{\gamma}((gh)^n) - f_{\gamma}(g^n) - f_{\gamma}(h^n)} \ = \
\abs{f_{\gamma}((gh)^n) - f_{\gamma}(h^n)} \ \leq \ 4n,\]
for a defect of at most $4$.  

\medskip \emph{Case II: $g$ and $h$ are hyperbolic, with positive
  overlap.} That is, $T_g$ and $T_h$ intersect in a segment, on which
$T_g$ and $T_h$ induce the same orientation. Let $e$ be an edge in
$T_g \cap T_h$, oriented coherently with $T_g$ and $T_h$. Let $v$ be
the terminal vertex of $e$. Since $e \in T_h$, the edges $h^{-1}e$ and
$e$ are coherently oriented and $\alpha = [h^{-1}v, v]$ is a
fundamental domain for $h$. Similarly, $e$ and $ge$ are coherently
oriented and $\beta = [v, gv]$ is a fundamental domain for $g$.

Let all the edges of $\alpha$ and $\beta$ be given orientations from
$T_h$ and $T_g$ respectively. 
Since $g$ and $h$ both move $e$ in the same direction (that is, into the
same component of $T - \{e\}$), $e$ separates $h^{-1}e$ from $ge$. It
follows that the edges of $\alpha$ and of $\beta$ are all coherently
oriented in $T$. Hence $\alpha$ and $\beta$ do not overlap, and $\alpha
\cdot \beta = [h^{-1}v, gv]$ is a fundamental domain for $gh$. 

With a total of four junctures (one for $h$, one for $g$, two for $gh$),
the estimate given by Lemma \ref{junctures} becomes 
\[ \abs{f_{\gamma}((gh)^n) - f_{\gamma}(g^n) - f_{\gamma}(h^n)} \ \leq \
4n,\] 
for a defect of at most $4$. 

\medskip \emph{Case III: $g$ and $h$ are hyperbolic, with negative
  overlap.} That is, $T_g$ and $T_h$ intersect in a segment, on which
$T_g$ and $T_h$ induce opposite orientations. Let $\Delta$ be the
length (possibly infinite) of $T_g \cap T_h$. There are several
sub-cases, according to the relative sizes of $\abs{g}$, $\abs{h}$,
and $\Delta$.

\medskip
\emph{Sub-case III-A: $\Delta \leq \abs{g}, \abs{h}$, not all three
  numbers equal.} Let $\rho$ be the segment $T_h \cap T_g$, oriented
coherently with $T_h$. There is a fundamental domain for $h$ of the form
$\alpha \cdot \rho$, and similarly, a fundamental domain for $g$ of the
form $\overline{\rho} \cdot \beta$; see Figure \ref{fig:III-A}. Then
$\alpha \cdot \beta$ is a fundamental domain for $gh$. (By assumption, at
least one of $\alpha$, $\beta$ is a non-trivial segment, and $gh$ is
hyperbolic.)  

\begin{figure}[ht]
\labellist
\small\hair 2pt
\pinlabel {$\alpha$} [tl] at 102 21
\pinlabel {$\beta$} [tr] at 137 19
\pinlabel {$\rho$} [r] at 112 50
\pinlabel {$T_h$} [l] at 149 97
\pinlabel {$T_g$} [l] at 183 2
\pinlabel {$T_{gh}$} [l] at 318 20
\endlabellist
\centering
\includegraphics{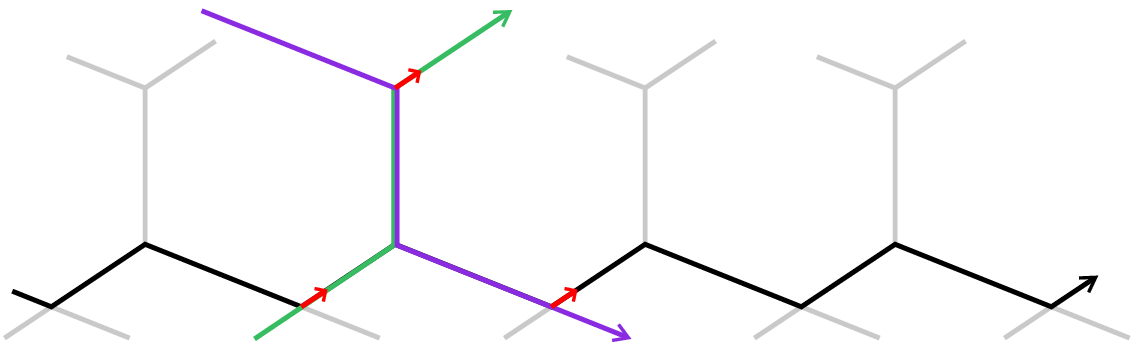}
\caption{Case III-A. $T_h$ is green, $T_g$ is purple, $T_{gh}$ is
  black. The red edges are of the form $e$, $he$, and 
  $ghe$.}\label{fig:III-A} 
\end{figure}

The quotient circuits have at most six junctures: two for $g$, two for
$h$, and two for $gh$. Lemma \ref{junctures} leads to an estimate
\[ \abs{f_{\gamma}((gh)^n) - f_{\gamma}(g^n) - f_{\gamma}(h^n)} \ \leq \
6n,\]
for a defect of at most $6$. 

\medskip
\emph{Sub-case III-B: $\abs{g} = \abs{h} \leq \Delta$.} In this case,
we show that $gh$ is elliptic. Let $\alpha \subset T_h \cap T_g$ be a fundamental
domain for $h$. Then $\overline{\alpha}$ is a fundamental domain for $g$, 
and $gh$ fixes the initial endpoint of $\alpha$. With two junctures in
total, we obtain the estimate
\[ \abs{f_{\gamma}((gh)^n) - f_{\gamma}(g^n) - f_{\gamma}(h^n)} \ = \
\abs{- f_{\gamma}(g^n) - f_{\gamma}(h^n)} \ \leq \ 2n,\]
for a defect of at most $2$. 

\medskip
\emph{Sub-case III-C: $\abs{h} < \abs{g}, \Delta$.} There is a
simplicial fundamental domain $\alpha$ for $h$ such that if $e$ is the
initial edge 
of $\alpha$, then $\alpha \cdot he$ is contained in $T_h \cap T_g$. Then,
there is a fundamental domain for $g$ of the form $\overline{\alpha}
\cdot \beta$; see Figure \ref{fig:III-CD}. By considering the location
of $ghe$, one finds that $\beta$ is a fundamental domain for $gh$. The
three circuits have a total of four junctures, and we obtain 
\[ \abs{f_{\gamma}((gh)^n) - f_{\gamma}(g^n) - f_{\gamma}(h^n)} \ \leq \
4n,\]
for a defect of at most $4$. 

\begin{figure}[ht]
\labellist
\small\hair 2pt
\pinlabel {$\alpha$} [b] at 58 34
\pinlabel {$\beta$} [tl] at 26 23
\pinlabel {$\alpha$} [bl] at 185 39
\pinlabel {$\beta$} [t] at 227 25
\pinlabel {$e$} [tl] at 48 14
\pinlabel {$he$} [tl] at 89 14
\pinlabel {$ghe$} [b] at 7 30
\pinlabel {$e$} [t] at 165 37
\pinlabel {$he$} [b] at 248 48
\pinlabel {$ghe$} [b] at 221 46.5
\pinlabel {$T_h$} [tl] at 128 57
\pinlabel {$T_g$} [r] at 2 3
\pinlabel {$T_h$} [tl] at 290 57
\pinlabel {$T_g$} [r] at 164 3
\endlabellist
\centering
\includegraphics{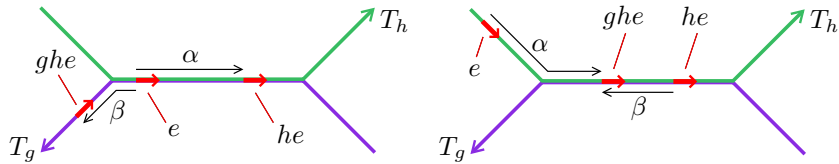}
\caption{Cases III-C (left) and  III-D (right). $T_h$ is green, $T_g$ is
  purple.}\label{fig:III-CD} 
\end{figure}

\medskip
\emph{Sub-case III-D: $\abs{g} < \abs{h}, \Delta$.} Fundamental domains 
$\beta$, $\alpha$, and $\alpha \cdot \overline{\beta}$ 
for $g$, $gh$, and $h$ respectively can be constructed in a similar
fashion as in Case 
III-C; see Figure~\ref{fig:III-CD}.  Alternatively, this case reduces to
Case III-C, replacing $g$ and $h$ by $h^{-1}$ and $g^{-1}$ respectively.  


\medskip
\emph{Case IV: $g$ is elliptic, $h$ is hyperbolic, $T_g \cap T_h \not=
  \emptyset$.} If $T_g \cap T_h$ contains an edge $e$, let $\alpha
\subset T_h$ be the fundamental domain starting with $h^{-1}e$. Then
$\alpha$ is also a fundamental domain for $gh$. This leads to an estimate 
\[ \abs{f_{\gamma}((gh)^n) - f_{\gamma}(g^n) - f_{\gamma}(h^n)} \ = \
\abs{f_{\gamma}((gh)^n) - f_{\gamma}(h^n)} \ \leq \ 2n,\] 
and a defect of at most $2$. 

If $T_g \cap T_h$ is a single vertex $v$, let $e\in T_h$ be the
coherently oriented edge with initial vertex $v$, and let $\alpha$ be the
fundamental domain $[h^{-1}v, v]$. If $ge \not\in \alpha$ then $\alpha$
is a fundamental domain for $gh$ also, and we obtain a defect of at most
$2$ as above. 

So now assume that $ge \in \alpha$, i.e.\ that $ge$ separates $h^{-1}v$
from $v$. Note that $h^{-1}e$ and $ge$ are not coherently oriented, so
the characteristic subtree $T_{gh}$ will not contain these edges. 

We have that $gh(\alpha) \cap \alpha$ contains the edge $ge$. Consider
the length of $gh(\alpha) \cap \alpha$. If this length is $\abs{\alpha}/2$ or
greater, then $gh$ fixes the midpoint of $\alpha$. Then
\[ \abs{f_{\gamma}((gh)^n) - f_{\gamma}(g^n) - f_{\gamma}(h^n)} \ = \
\abs{- f_{\gamma}(h^n)} \ \leq \ n,\] 
giving a defect of at most $1$. 

Otherwise, there is a subsegment $\beta \subset \alpha$, centered on the
midpoint of $\alpha$, of maximal size so that $\beta$ does not overlap
$gh\beta$. We can write $\alpha$ as a concatenation $\alpha_1 \cdot \beta
\cdot \alpha_2$, where $\alpha_2 = gh\overline{\alpha}_1$. See Figure
\ref{fig:IV}. 
\begin{figure}[ht]
\labellist
\small\hair 2pt
\pinlabel {$\alpha_1$} [b] at 46 13
\pinlabel {$\alpha_2$} [b] at 117 31
\pinlabel {$gh \alpha_2$} [b] at 46 49
\pinlabel {$\beta$} [tl] at 83 19
\pinlabel {$gh \beta$} [tr] at 81 36
\pinlabel {$T_h$} [l] at 187 26
\pinlabel {$T_g$} [c] at 147 14
\pinlabel {$T_{gh}$} [r] at 82 73
\endlabellist
\centering
\includegraphics{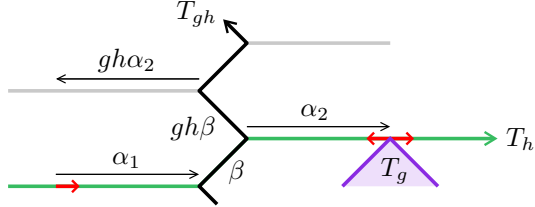}
\caption{Case IV. The element $gh$ takes $\alpha_1$ to
  $\overline{\alpha}_2$. Left to right, the red edges are $h^{-1}e$, $ge$, and
$e$.}\label{fig:IV}   
\end{figure}
Now $\beta$ is a fundamental domain for $gh$, and we have a total of
four junctures (three in $T_h / \langle h\rangle$ and one in
$T_{gh}/\langle gh \rangle$). Thus we have 
\[ \abs{f_{\gamma}((gh)^n) - f_{\gamma}(g^n) - f_{\gamma}(h^n)} \ = \
\abs{f_{\gamma}((gh)^n) - f_{\gamma}(h^n)} \ \leq \ 4n,\] 
and a defect of at most $4$. 

\medskip
\emph{Case V: $h$ is elliptic, $g$ is hyperbolic, $T_g \cap T_h \not=
  \emptyset$.} This case is covered by Case IV, replacing $g$ and $h$ by
$h^{-1}$ and $g^{-1}$ respectively. 

\medskip
\emph{Case VI: the remaining cases.} If $g$ and $h$ are hyperbolic and
$T_g$ and $T_h$ intersect in one point, then the configuration closely
resembles the first one discussed in Case I, except that the copies of
$\rho$ have been shrunk to have length zero. That is, there are
fundamental domains $\alpha$ and $\beta$ for $h$ and $g$ respectively,
such that $\alpha \cdot \beta$ is a fundamental domain for $gh$. With
four junctures, we obtain 
\[ \abs{f_{\gamma}((gh)^n) - f_{\gamma}(g^n) - f_{\gamma}(h^n)} \ \leq \
4n,\]
for a defect of at most $4$. 

Lastly, if $g$ and $h$ have a common fixed point, then
\[\abs{f_{\gamma}((gh)^n) - f_{\gamma}(g^n) - f_{\gamma}(h^n)} \ = \ 0\] 
for all $n$. 
\end{proof}

\begin{remark}\label{rem:rtree}
  The functions $f_\gamma$ and $h_\gamma$ can be defined in the more
  general setting of a group acting on an $\R$--tree.  The proof of
  Theorem~\ref{defect} goes through in this setting, with only
  superficial modifications (essentially, removing any mention of
  \emph{edges}, and using small segments instead). 
\end{remark}

\subsection*{Well-aligned elements}
We now consider elements $g \in G$ for which we can find a segment
$\gamma \subset T$ such that $h_\gamma(g) = 1$.

\begin{definition}\label{def:well-aligned}
  Given a $G$--tree $T$, a hyperbolic element $g \in G$ is
  \emph{well-aligned} if there does not exist an element $h\in G$ such
  that $ghgh^{-1}$ fixes an edge of $T_g$. 
  This property is the $G$--tree analogue of the double coset
  condition from \cite[Theorem~D]{CF}.  
\end{definition}

\begin{theorem}\label{th:well-aligned}
  Suppose $G$ acts on a simplicial tree $T$. If $g\in G$ is
  well-aligned then $\scl(g) \geq 1/12$.
\end{theorem}

\begin{proof}
  Let $\gamma = [x, gx] \subset T_g$ be a fundamental domain for $g$
  where $x$ is \emph{not} a vertex of $T$. We know that
  $c_{\gamma}(g^n) = n$ for all $n$. If $c_{\overline{\gamma}}(g^n)>
  0$ for some $n$, then there is a copy of $\overline{\gamma}$ in
  $T_g$. That is, there is an element $h$ such that $h \gamma$ lies in
  $T_g$ with the opposite orientation. So $h(T_g) = T_{hgh^{-1}}$ has
  negative overlap with $T_g$ along a segment containing $h
  \gamma$. The element $ghgh^{-1}$ fixes one of the endpoints of
  $h\gamma$, since $g$ and $hgh^{-1}$ shift it in opposite directions
  inside $T_g \cap T_{hgh^{-1}}$. This endpoint is in the interior of
  an edge $e \subset T_g$, and so $ghgh^{-1}$ fixes $e$. Hence, if $g$
  is well-aligned, we must have $h_{\gamma}(g) = 1$. By
  Theorem~\ref{defect}, $h_{\gamma}$ is a homogeneous quasimorphism
  with defect at most $6$, and so Proposition \ref{bavard} implies
  that $\scl(g) \geq 1/12$.
\end{proof}

This bound is in fact \emph{optimal}.  Both in HNN extensions and in
amalgamated free products, there are examples of elements $g$ with
$\scl(g)=1/12$ that are well-aligned with respect to the action on the
associated Bass--Serre tree, as we now explain.  This answers
Question~8.4 from \cite{CF}. 

\begin{theorem}\label{th:optimal}
  Let $g = tat^{-1}a \in BS(2,3)$ and let $T$ be the Bass--Serre tree
  associated to the splitting of $BS(2,3)$ as an HNN extension $\langle a
  \rangle*_{\langle ta^2t^{-1} = a^3\rangle}$.  Then $g$ is
  well-aligned and $\scl(g) = 1/12$.  In particular, the bound in
  Theorem~\ref{th:well-aligned} is optimal.
\end{theorem}

\begin{proof}
  Denote the vertex of $T$ stabilized by $\langle a \rangle$ by $v_0$
  and let $v_1 = tv_0$.  The vertices along the axis of $g$ are: $\{
  g^nv_0, g^nv_1 \}_{n \in \Z}$.

  If $ghgh^{-1}$ fixes an edge $e \subset T_g$, then we also see that
  $hgh^{-1}g$ fixes $g^{-1}e \subset T_g$.  Replacing $h$ by $hg^k$ for
  some $k$ (which does not affect $hgh^{-1}$), we can arrange that $h$
  fixes a vertex of $T_g$. By further replacing $h$ by a conjugate $g^k h
  g^{-k}$, we can arrange that the vertex fixed by $h$ is either $v_0$ or
  $v_1$; the elements $ghgh^{-1}$ and $hgh^{-1}g$ still fix edges of
  $T_g$. 

  First assume that $h$ fixes $v_0$, and so $h = a^r$ for some $r
  \in \Z$.  In this case \[ ghgh^{-1} = tat^{-1}a^{1 +
    r}tat^{-1}a^{1-r}.\] If $ghgh^{-1}$ is elliptic (which it
  necessarily is if it fixes an edge), then this expression cannot be
  cyclically reduced (Remark~\ref{cyc-red}).  Hence we find that $r
  \equiv \pm 1 \mod 3$.  If $r \equiv 1 \mod 3$, then $hgh^{-1}g =
  a^5$; if $r \equiv -1 \mod 3$ then $ghgh^{-1} = a^5$.  In either
  case, the element does not fix an edge in $T$, giving a contradiction. 

  Similarly, if $h$ fixes $v_1$, then we have $h = ta^rt^{-1}$ for some
  $r \in \Z$, and so
  \[ hgh^{-1}g = ta^{1+r}t^{-1}ata^{1-r}t^{-1}a.\] Again, this
  expression cannot be cyclically reduced if $hgh^{-1}g$ is elliptic,
  and so $r \equiv 1 \mod 2$.  Again, we find that $hgh^{-1}g = a^5$,
  giving a contradiction for the same reason as above.  Therefore $g$ is
  well-aligned as claimed.

  Finally, $\scl(tat^{-1}a) = 1/12$ by Proposition~\ref{prop:length2}. 
\end{proof}

The bound in Theorem~\ref{th:well-aligned} is still optimal if one
restricts to amalgamated free products.  In the free product
$\Z/2\Z*\Z/3\Z\cong \PSL(2,\Z)$, no nontrivial element fixes an edge of
the associated Bass--Serre tree, so every hyperbolic element that is not
conjugate to its inverse is well-aligned.  The group $\PSL(2,\Z)$ has a
finite index free subgroup, and therefore stable commutator length can be
computed in this group by using a relationship between stable commutator
length in a group and a finite index subgroup from \cite{Calegari:scl}
together with Calegari's algorithm for computing stable commutator length
in free groups \cite{Calegari:free}.  This is described explicitly in
\cite{Louwsma:thesis}.  The element
$\bigl(\begin{smallmatrix}1&0\\1&1\end{smallmatrix}\bigr)$ is an example
of an element that has stable commutator length $1/12$. 

\subsection*{Acylindrical trees} 
We conclude this section by adding a moderate restriction,
\emph{acylindricity}, to the tree action. We can then say something about
hyperbolic elements that are not necessarily well-aligned. 

Acylindricity 
has been used previously in the context of counting quasimorphisms on
Gromov-hyperbolic spaces, cf.~\cite{CF}.  For a tree, the definition is
particularly simple to state.  A group acts \emph{$K$--acylindrically} on
a tree if the stabilizer of any segment of length $K$ is trivial. 

\begin{theorem}\label{cor:acylindrical}
  Suppose $G$ acts $K$--acylindrically on a tree $T$ and let $N$ be
  the smallest integer greater than or equal to $\frac{K}{2} + 1$. 
\begin{enumerate}
\item \label{g1} If $g \in G$ is hyperbolic then either $\scl(g) = 0$ or
  $\scl(g) \geq 1/12N$. 
\item \label{g2} If $g \in G$ is hyperbolic and $\abs{g} \geq K$ then
  either $\scl(g) = 0$ or $\scl(g) \geq 1/24$. 
\end{enumerate}
In both cases, $\scl(g) = 0$ if and only if $g$ is conjugate to
  $g^{-1}$.
\end{theorem}

\begin{proof}
First note that if $\abs{g} = 1$ then 
a fundamental domain for $g$ maps to a single loop in the quotient graph
of $T$, which implies that $g$ has infinite order in the abelianization
of $G$, and $\scl(g) = \infty$. (In fact, the same conclusion holds
whenever $\abs{g}$ is odd.) Thus we may assume that $\abs{g} \geq 2$.  

  Observe that if for some $h \in G$, we have $\abs{T_g
    \cap T_{hgh^{-1}}} \geq K + \abs{g}$ where $g$ and $hgh^{-1}$
  shift in opposite directions, then $ghgh^{-1}$ fixes a segment of
  length $K$ and hence $g = hg^{-1}h^{-1}$.  In particular, $\scl(g) =
  0$. 

For \eqref{g1} note that $\abs{g^N} = N \abs{g} \geq \frac{K}{2}\abs{g} +
\abs{g} \geq K + \abs{g}$. Taking $\gamma$ to be a fundamental domain for
$g^N$, if $h_{\gamma}(g^N) < 1$ then there is an $h$ as above and
$g = hg^{-1}h^{-1}$, $\scl(g) = 0$. Otherwise, $h_{\gamma}(g^N)=1$ and
$\scl(g) = \frac{1}{N}\scl(g^N) \geq 1/12N$ by Theorem \ref{defect} and
Proposition \ref{bavard}. 

For \eqref{g2} let $N = 2$ and apply the same reasoning: $\abs{g^2}  =
2\abs{g} \geq K + \abs{g}$, and either $\scl(g) = 0$ or $\scl(g) =
\frac{1}{2} \scl(g^2) \geq 1/24$. 
\end{proof}

\section{The gap theorem}\label{sec:gap}

In this section we consider $G$--trees of a particular form, for which
we can improve upon the ``well-aligned'' condition in Theorem
\ref{th:well-aligned} without any trade-off in the lower bound of $1/12$. 

Let $G$ be an HNN extension $A \, \ast_C$, with
stable letter $t$, such that the edge groups $C$ and $C^t$ are central in
$A$. Let $T$ be the Bass--Serre tree associated to this HNN extension. 
Such a tree has special properties, given below in Lemma \ref{rotate} and
Proposition \ref{stabilizer}.

This class of HNN extensions obviously includes the Baumslag--Solitar
groups. It is still true (cf.~Remark~\ref{t-exp}) that an element cannot
have finite stable commutator length if its $t$--exponent is non-zero. 

We use the following notation for a $G$--tree $T$: if $X\subset T$ is any
subset, then $G_X$ denotes the \emph{stabilizer} of $X$, which is the
subgroup of $G$ consisting of those elements that fix $X$ pointwise. 

\begin{lemma}\label{rotate} 
Suppose $S \subset T$ is a subtree and $h\in G$ fixes a vertex $v\in
S$. Then $G_{S} = G_{h S}$.  
\end{lemma}

\begin{proof}
Let $e$ be an edge in $S$ with endpoint $v$. Then $G_{S} \subset
G_e$. Since $G_e$ is central in $G_v$, $h$ commutes with $G_{S}$, and
therefore $G_{h S} = h G_{S} h^{-1} = G_{S}$. 
\end{proof}

Consider the $t$--exponent homomorphism $\phi \co G \to \Z$ (sending $t$
to $1$ and $A$ to $0$). There is an action of $\Z$ on $\R$ by integer
translations. Letting $G$ act on $\R$ via $\phi$, there is also a
$G$--equivariant map $F \co T \to \R$. This map is just the natural map from
$T$ to the universal cover of the quotient graph $T/G$. The action of an
element $g$ on $T$ projects by $F$ to a translation by $\phi(g)$ on
$\R$. We think of $F$ as a \emph{height function} on $T$. Then, the
elements of $t$--exponent zero act on $T$ by height-preserving
automorphisms. 

\begin{proposition}\label{stabilizer} 
If $S \subset T$ is a subtree and $\sigma \subset S$ is a finite
subtree such that $F(\sigma) = F(S)$ then $G_{\sigma} = G_{S}$. 
\end{proposition}

\begin{proof}
\emph{Case I: $\sigma$ is a segment mapped by $F$ injectively to $\R$.}  

First suppose that $S$ is a finite subtree. Let $\{S_i\}$ be the
subtrees obtained as the closures of the components of $S -
\sigma$. Then $G_{S} = G_{\sigma} \cap \bigcap_i G_{S_i}$. Fixing $i$, we
will prove by induction on the number of edges of $S_i$ that $G_{\sigma}
\subset G_{S_i}$. It then follows that $G_{\sigma} = G_{S}$. 

The base case is that $S_i$ is a single edge $e$ with one vertex $v$ on
$\sigma$. Since $F(S_i) \subset F(\sigma)$, there is an edge $e'$ on
$\sigma$ with endpoint $v$ and an element $h \in G_v$ taking $e$ to
$e'$. Indeed, edges incident to $v$ are in the same $G_{v}$--orbit if and
only if they have the same image under $F$.  By Lemma \ref{rotate} we
have that $G_{S_i} = G_{e'} \supset G_{\sigma}$. 

For larger $S_i$, let $e \in S_i$ be the edge with endpoint $v \in
\sigma$. Again there is an edge $e'$ on $\sigma$ with endpoint $v$ and an
element $h \in G_v$ taking $e$ to $e'$. Again, $G_{S_i} = G_{h S_i}$ by
Lemma \ref{rotate}. But now $h S_i = e' \cup S_i'$ where $S_i'$ has fewer
edges than $S_i$. By induction, $G_{\sigma} \subset G_{S_i'}$. Since
$G_{\sigma} \subset G_{e'}$, we now have $G_{\sigma} \subset (G_{e'} \cap
G_{S_i'}) = G_{S_i}$. 

Now consider an arbitrary subtree $S$. We need to show that $G_{\sigma}$
fixes $S$ pointwise. But every point $x$ in $S$ is in a finite subtree $S'$
containing $\sigma$, and $G_{\sigma}$ fixes $S'$ pointwise; hence
$G_{\sigma}$ fixes $x$. 

\emph{Case II: $\sigma$ is an arbitrary finite subtree of $S$.} Fixing
the image $F(\sigma)$, we proceed by induction on the number of edges of
$\sigma$. The base case is when this number is smallest, namely the
length of $F(\sigma)$. Then Case I applies. If there are more edges than
this, there must be a vertex $v\in \sigma$ and a pair of edges $e_0,
e_1\in \sigma$ incident to $v$, with $F(e_0) = F(e_1)$. 

Decompose $S$ into two subtrees $S = S_0 \cup S_1$ with $S_0 \cap S_1 =
\{v\}$ and $e_0 \in S_0$, $e_1 \in S_1$. Let $\sigma_i = S_i \cap
\sigma$. There is an element $h \in G_v$ such that $he_0 = e_1$. Let  
$\sigma' = h\sigma_0 \cup \sigma_1$ and $S' = hS_0 \cup S_1$. Note
that $\sigma'$ has fewer edges than $\sigma$. Also, $F(\sigma) =
F(\sigma')$ and $F(S) = F(S')$, and $G_{\sigma} = G_{\sigma_0} \cap
G_{\sigma_1} = G_{h\sigma_0} \cap G_{\sigma_1} = G_{\sigma'}$ by Lemma
\ref{rotate}. Similarly, $G_{S} = G_{S_0} \cap G_{S_1} = G_{hS_0} \cap
G_{S_1} = G_{S'}$. By the induction hypothesis, $G_{\sigma'} = G_{S'}$,
and therefore $G_{\sigma} = G_S$. 
\end{proof}

Now consider a hyperbolic element $g$ with $t$--exponent zero. The axis
$T_g$ has the property that $F(T_g)$ is a finite interval. To see this,
let $\gamma$ be a fundamental domain, and note that $T_g = \bigcup_n g^n
\gamma$. The $t$--exponent condition implies that $F(g^n \gamma) =
F(\gamma)$ for all $n$, and hence $F(T_g) = F(\gamma)$. 

\begin{definitions}
We call a vertex $v$ on $T_g$ \emph{extremal} if $F(v)$ is an endpoint of
$F(T_g)$. A segment $\sigma \subset T_g$ is \emph{stable} if $F(\sigma) =
F(T_g)$ and $\sigma$ contains no extremal vertex in its interior
(equivalently, no proper subsegment $\sigma'$ satisfies $F(\sigma') =
F(T_g)$). See Figure \ref{fig:stable}. Note that if $\sigma$ and $\tau$
are stable segments, then they do not overlap, unless they are equal. 
\begin{figure}[ht]
\labellist
\small\hair 2pt
\pinlabel {$T_g$} [c] at 75 19
\pinlabel {$F$} [b] at 312 30
\endlabellist
\centering
\includegraphics{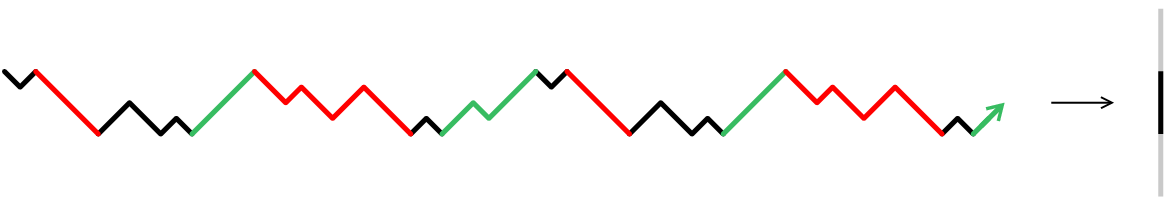}
\caption{Stable segments along $T_g$, in green and red.}\label{fig:stable}   
\end{figure}

The natural orientation of $T_g$ defines a linear ordering $<_g$ on
the stable segments of $T_g$. The ``larger''
end is the attracting end of $T_g$; that is, $\sigma <_g g\sigma$ always
holds. We say that $\sigma \leq_g \tau$ if $\sigma <_g \tau$ or $\sigma =
\tau$. 
\end{definitions}

\begin{remark}\label{cleanremark}
If $\gamma$ is a fundamental domain for $g$ whose endpoints are extremal,
then every stable segment either does not overlap with $\gamma$ or is 
contained in $\gamma$. Moreover, $\gamma$ contains a copy of every stable
segment. (Being a fundamental domain, it overlaps with a copy of every
non-trivial segment in $T_g$.)  If $\gamma$ is a fundamental domain that
starts with a stable segment then its endpoints are extremal, as the
endpoints of a stable segment are extremal. 
\end{remark}

Proposition \ref{stabilizer} immediately implies: 

\begin{corollary}\label{cor:stable}
If $\sigma \subset T_g$ is stable then $G_{\sigma} = G_{T_g}$. 
\end{corollary}

The main technical result of this section is: 

\begin{theorem}\label{infiniteoverlap}
Let $G = A \, \ast_C$ with stable letter $t$, and $C, C^t$ central in
$A$. Let $g\in G$ be a hyperbolic element with $t$--exponent zero. Then
either: 
\begin{enumerate}
\item there is a fundamental domain $\gamma$ for $g$ such that
  $h_{\gamma}(g) = 1$, or \label{one} 
\item there is an element $h$ such that $h(T_g) =
  \overline{T}_g$. \label{two} 
\end{enumerate}
\end{theorem}

Conclusion \eqref{one} implies by Proposition~\ref{bavard} and
Theorem~\ref{defect} that $\scl(g) \geq 1/12$. 

\begin{proof}
  Let $\alpha \subset T_g$ be a stable segment and let $\gamma$ be the
  fundamental domain for $g$ that starts with $\alpha$. Note that
  $\gamma$ has extremal endpoints. If
  $h_{\gamma}(g) < 1$ then there is an element $h$ such that $h\gamma$
  lies in $T_g$ with the opposite orientation and overlaps with
  $\alpha$. Note that $h$ fixes a point in $\gamma$.  In
  particular $h$ is elliptic, and hence acts as a height-preserving
  automorphism of $T$.  Now $h\gamma$ is a fundamental domain for
  $g^{-1}$ with extremal endpoints, and so $h\gamma$ contains $\alpha$
  (which is stable for $T_{g^{-1}}$ as well as for $T_g$).

  The segment $\beta = h^{-1}\alpha$ is a stable segment for $T_g$
  contained in $\gamma$.  Clearly $\alpha \leq_g \beta$, and as the
  endpoints of $\alpha$ have different heights, $\alpha \neq \beta$; 
  therefore $\alpha <_g \beta$.  Note that $h\alpha = \beta$, since
  $h$ acts as a reflection on the segment $\gamma \cap h\gamma$.
  Hence the element $h^2$ fixes the stable segment $\alpha$. Therefore,
  by Proposition \ref{stabilizer}, $h^2$ fixes $T_g \cup h(T_g)$. That
  is, $h$ acts as an involution on this entire subtree of $T$.

\begin{claim}
  If there is a stable segment $\rho
  \subset T_g \cap h(T_g)$ such that either $\rho <_g \alpha$ or
  $\beta <_g \rho$, then conclusion \eqref{two} holds.
\end{claim}

  \begin{proof}[Proof of Claim]
     Since $h$ acts as a reflection on the segment $T_g \cap h(T_g)$, if
     $\rho \subset T_g \cap h(T_g)$ and $\rho <_g \alpha$, then $h\rho
     \subset T_g \cap h(T_g)$ and $\beta <_g h\rho$. Thus we only need to
     verify the claim in the $\beta <_g \rho$ case. 

    Let $\sigma$ be the $<_g$--smallest stable segment in $h\gamma$.
    Observe that $\sigma \subset T_g \cap h(T_g)$.
    The translate $g\sigma$ is the $<_g$--smallest stable segment in
    $gh\gamma$, which has a common endpoint with $\beta$.  Hence
    $g\sigma \leq_g \tau$ for any $\tau$ satisfying $\beta <_g \tau$.
    In particular, $g\sigma \leq_g \rho$.  Since $\beta,\rho \subset
    T_g \cap h(T_g)$ and $\beta <_g g\sigma \leq_g \rho$, it follows
    that $g\sigma \subset T_g \cap h(T_g)$.   

    Note that $h(T_g) = T_{hgh^{-1}}$, and $hgh^{-1}$ takes $g\sigma$
    to $\sigma$. Thus $ghgh^{-1}$ fixes $g\sigma$. Since this is a
    stable segment, $ghgh^{-1}$ must fix all of $T_g$ by
    Corollary~\ref{cor:stable}. This implies that $hgh^{-1}$ acts on
    $T_g$ as a translation, of the same amplitude but opposite
    direction as $g$. Hence $T_{hgh^{-1}} = \overline{T}_g$.
  \end{proof}

  Returning to the proof of Theorem \ref{infiniteoverlap}, 
  assume that the Claim does not apply. Then $\alpha$ and $\beta$ are
  the $<_g$--smallest and $<_g$--largest stable segments in $T_g \cap
  h(T_g)$ respectively. It follows that $\beta$ is also the
  $<_g$--largest stable segment in $\gamma$; otherwise, if $\beta <_g
  \rho$ and $\rho \subset \gamma$, then $h\rho <_g \alpha$ and $h\rho
  \subset T_g \cap h(T_g)$, contradicting that $\alpha$ is smallest. 

  Note that $h$ takes stable segments to stable segments, and does not take
  any stable segment to itself (since the endpoints have different
  heights). Hence the stable segments of $\gamma$ may be enumerated in
  order as $\alpha =   \alpha_1, \ldots, \alpha_n, \beta_n, \ldots,
  \beta_1 = \beta$ where $h$ interchanges $\alpha_i$ and $\beta_i$. 

Now let $\gamma'$ be the fundamental domain for $g$ starting with
$\beta_n$. Assuming that conclusion \eqref{one} does not hold, we have
$h_{\gamma'}(g) <1$, and so there is an elliptic element $k$ such that
$k\gamma'$ lies in $T_g$ with the opposite orientation and contains
$\beta_n$. 

The configuration of $T_g$, $k(T_g)$, $\beta_n$, and $k\beta_n$ is
exactly analogous to that of $T_g$, $h(T_g)$, $\alpha$, and $\beta$. In
particular, the Claim is applicable to this situation. If the Claim does
not apply, then we conclude as above that $\beta_n$ and $k\beta_n$ are
the $<_g$--smallest and $<_g$--largest stable segments in $T_g \cap k(T_g)$
and that $k\beta_n$ is the $<_g$--largest stable segment in $\gamma'$. 

The stable segments in $\gamma'$ are, in order: $\beta_n,
\ldots, \beta_1, g\alpha_1, \ldots, g\alpha_n$, and the element $k$
interchanges $\beta_i$ and $g\alpha_i$. Thus 
\[kh\alpha = kh\alpha_1 = k\beta_1 = g\alpha_1 = g\alpha.\]
Since $kh$ and $g$ agree on the stable
segment $\alpha$, they agree on all of $T_g \cup h(T_g) \cup k(T_g)$, by
Proposition \ref{stabilizer}. Similarly, $h$ and $k$ both act as
involutions on $T_g \cup h(T_g) \cup k(T_g)$. Now 
\[ h g^{-1} \beta = h (kh)^{-1} \beta = h h k \beta = 
k\beta = g\alpha,\]
which implies that $g\alpha \subset T_g \cap h(T_g)$. However, $\beta
<_g g\alpha$ and $\beta$ is the $<_g$--largest stable segment in $T_g
\cap h(T_g)$. This contradiction establishes the theorem. 
\end{proof}

The next proposition concerns conclusion \eqref{two} in Theorem
\ref{infiniteoverlap}. It is a variant of the observation that if an
element is conjugate to its inverse, then it has $\scl$ zero. 

\begin{proposition}\label{sclzero}
Suppose $G$ acts on a tree $T$ and $\scl$ vanishes on the elliptic
elements of $G$. If $g$ is hyperbolic and there is an element $h$ such
that $h(T_g) = \overline{T}_g$, then $\scl(g) = 0$. 
\end{proposition}

\begin{proof}
Since $h(T_g) = T_{hgh^{-1}}$, the element $ghgh^{-1}$ fixes
$T_g$ pointwise. Similarly, $g^n h g^n h^{-1}$ fixes $T_g$ for every
$n$. Thus there are elliptic elements $a_n$ such that $g^n hg^nh^{-1} 
= a_n$. This equation can be realized by a surface of genus zero and
three boundary components, labeled by $g^n$, $g^n$, and $a_n^{-1}$
respectively. Lemma~\ref{fill-with-zero} now implies that
  \begin{equation*}
    \scl(g) \ \leq \ \frac{1}{4n} + \frac{\scl(a_n^{-1})}{2n}.
  \end{equation*}
Hence $\scl(g) \leq 1/4n$ for all $n > 0$. 
\end{proof}

\begin{theorem}[Gap theorem]\label{th:gap}
  For every element $g \in BS(m, \ell)$, either $\scl(g) = 0$ or
  $\scl(g) \geq 1/12$. 
\end{theorem}

\begin{proof}
Every elliptic element $g$ is conjugate to a power of $a$, and therefore
$\scl(g) = 0$ for elliptic elements by Lemma \ref{scl-a}. If $g$ is hyperbolic then Theorem
\ref{infiniteoverlap} and Proposition \ref{sclzero} imply that $\scl(g) =
0$ or $\scl(g) \geq 1/12$. 
\end{proof}

\begin{remark}
  If $m$ and $\ell$ are both odd, then conclusion \eqref{two} of
  Theorem \ref{infiniteoverlap} can never occur, since the element $h$
  would fix a vertex of $T_g$ and exchange two adjacent edges,
  yielding an element of order two in $\Z/m\Z$ or
  $\Z/\ell\Z$. Therefore, $\scl(g) \geq 1/12$ for \emph{every}
  hyperbolic element $g$ in $BS(m, \ell)$.

  If either $m$ or $\ell$ is even, say $m = 2k$, then for $g = ta^k
  t^{-1} a^r ta^k t^{-1} a^s \in BS(m,\ell)$ where $r+s=\ell$ we have
  $\scl(g) = 0$.  Indeed, taking $h = ta^{-k}t^{-1}$ one checks that
  $ghgh^{-1} = a^{4\ell} \in \langle a^\ell \rangle = G_{T_g}$. Thus
  $h(T_g) = \overline{T}_g$ and so by Proposition~\ref{sclzero} we
  have $\scl(g) = 0$.
\end{remark}

\bibliographystyle{amsplain}
\bibliography{bib}

\end{document}